\documentclass[11pt,a4paper]{article}
\usepackage{tikz}
\usepackage{subfigure}
\usepackage[english]{babel}
\usepackage{graphicx}
\usepackage{extarrows}
\usepackage{lineno}
\usepackage{indentfirst, latexsym, bm, amssymb}

\begin{document}

\newtheorem{theorem}{Theorem}[section]
\newtheorem{problem}[theorem]{Problem}
\newtheorem{corollary}[theorem]{Corollary}
\newtheorem{observation}[theorem]{Observation}
\newtheorem{definition}[theorem]{Definition}
\newtheorem{conjecture}[theorem]{Conjecture}
\newtheorem{question}[theorem]{Question}
\newtheorem{lemma}[theorem]{Lemma}
\newtheorem{proposition}[theorem]{Proposition}
\newtheorem{example}[theorem]{Example}
\newenvironment{proof}{\noindent {\bf
Proof.}}{\hfill $\square$\par\medskip}
\newcommand{\remark}{\medskip\par\noindent {\bf Remark.~~}}
\newcommand{\pp}{{\it p.}}
\newcommand{\de}{\em}

\title{Extremal graphs for  edge blow-up of graphs\thanks{This work is supported by the National Natural Science Foundation of China (No. 11901554) and Science and Technology Commission
of Shanghai Municipality (No. 18dz2271000, 19jc1420100)}}
\author{Long-Tu Yuan  \\
{\small School of Mathematical Sciences and Shanghai Key Laboratory of PMMP}\\
{\small East China Normal University}\\
{\small 500 Dongchuan Road, Shanghai, 200240, P.R. China  }\\
{\small Email: ltyuan@math.ecnu.edu.cn}\\
}
\date{}
\maketitle
\begin{abstract}
Given a graph $H$ and an integer $p$, the  edge blow-up $H^{p+1}$ of $H$ is the graph obtained from replacing each edge in $H$ by a clique of order $p+1$ where the new vertices of the cliques are all distinct. The Tur\'{a}n numbers for edge blow-up of matchings were first studied by Erd\H{o}s and Moon. In this paper, we determine the range of the Tur\'{a}n numbers for edge blow-up of all bipartite graphs. Moreover, we characterize the extremal graphs for edge blow-up of all non-bipartite graphs. Our results also extend several known results, including the Tur\'{a}n numbers for edge blow-up of stars, paths and cycles. The method we used can also be applied to find a family of counter-examples to a conjecture posed by Keevash and Sudakov in 2004 concerning the maximum number of edges not contained in any monochromatic copy of $H$ in a $2$-edge-coloring of $K_n$.
\end{abstract}


\vskip 0.5cm

\section{Introduction}
\noindent Given a family of graphs $\mathcal{H}$, we say a graph $G$ is $\mathcal{H}$-free ($H$-free if $\mathcal{H}=\{H\}$) if $G$ does not contain any copy of $H\in \mathcal{H}$ as a subgraph.
The {\it Tur\'{a}n number} of a family of graphs $\mathcal{H}$, denote as ex$(n,\mathcal{H})$, is the maximum number of edges in an $\mathcal{H}$-free graph $G$ of order $n$.
Denote by EX$(n,\mathcal{H})$ the set of $\mathcal{H}$-free graphs on $n$ vertices with ex$(n,\mathcal{H})$ edges and call a graph in EX$(n,\mathcal{H})$ an extremal graph for $H$.
We simply use ex$(n,H)$ and  EX$(n,H)$  instead of ex$(n,\{H\})$ and  EX$(n,\{H\})$ respectively if $\mathcal{H}=\{H\}$.

In 1941, Tur\'{a}n \cite{turan1941} proved that the unique extremal graph without containing a clique on $p+1\geq 3$ vertices is
the complete $p$-partite graph on $n$ vertices which is balanced, in that the partite sizes are as equal as possible. This balanced complete $p$-partite graph on $n$ vertices is the Tur\'{a}n graph $T_p(n)$ and let $t_p(n)=e(T_p(n))$ be the number of edges of $T_p(n)$.

Later, in 1946, Erd\H{o}s and Stone \cite{erdHos1946} proved the following well-known theorem.
\begin{theorem}[Erd\H{o}s and Stone \cite{erdHos1946}]\label{erdos-stone}
For all integers $p\geq 2$ and $N\geq 1$, and every $\epsilon >0$, there exists an integer $n_0$ such that every graph with $n\geq n_0$ vertices and at least
\[t_{p-1}(n)+\epsilon n^2\]
edges contains $T_p(Np)$ as a subgraph.
\end{theorem}

Let $\mathcal{F}$ be a family of graphs, the {\it subchromatic number } $p(\mathcal{F})$ of $\mathcal{F}$ is defined by
\[p(\mathcal{F})=\min\{\chi(F):F\in \mathcal{F}\}-1,\]
where $\chi(F)$ is the chromatic number of $F$. The classical Erd\H{o}s-Stone-Simonovits theorem \cite{erdHos1946,erdHos1966} states that
$$\mbox{ex}(n,\mathcal{F})=\left(1-\frac{1}{p(\mathcal{F})}\right){n \choose 2}+o(n^2).$$

If $\mathcal{F}$ contains a bipartite graph, then $p(\mathcal{F})=1$ and ex$(n,\mathcal{F})=o(n^2) $. For this degenerate (bipartite) extremal graph problem, there is an excellent survey by F\"{u}redi and Simonovits \cite{Furedi2013}. Let $G$ be a graph with $\chi(G)=p+1\geq 3$. If there is an edge $e$ such that $\chi(G-\{e\})=p$, then we say that $G$ is {\it edge-critical}  and $e$ is a {\it critical edge}. The Tur\'{a}n numbers of  edge-critical graphs are determined when $n$ is sufficiently large. In 1968, Simonovits \cite{Simonovits1968} proved the following theorem.
\begin{theorem}[Simonovits \cite{Simonovits1968}]\label{Chromatic Critical Edge}
Let $G$ be an edge-critical  graph with $\chi(G)=p+1\geq 3$. Then there exists an $n_0$ such that if $n>n_0$ then $T_p(n)$ is the unique extremal graph for $G$ on $n$ vertices.
\end{theorem}

Although the Tur\'{a}n numbers of non-bipartite graphs are asymptotically  determined by Erd\H{o}s-Stone-Simonovits theorem, it is a challenge to determine the exact Tur\'{a}n functions for many non-bipartite graphs,  There are only few graphs whose Tur\'{a}n numbers were determined exactly, including edge-critical graphs \cite{Simonovits1968} and some specific graphs \cite{chen2003,Simonovits1974.1,Yuan2018}.

Given a graph $H$ and an integer $p\geq2$, the {\it edge blow-up} of $H$, denoted by $H^{p+1}$, is the graph obtained from replacing each edge in $H$ by a clique of order $p+1$ where the new vertices of the cliques are all distinct.
The subscript in the case of graphs indicates the number of vertices, e.g., denote by $P_{k}$ a path on $k$ vertices, $S_{k}$ a star on $k$ vertices and $K_{n_1,\ldots,n_p}$ the complete $p$-partite graph with partite sizes $n_1,\ldots,n_p$.
A {\it matching} in $G$ is a set of edges from $E(G)$, no two of which share a common vertex, and the matching number of $G$, denoted by $\nu(G)$, is the  number of edges in a maximum matching.
Accordingly, we denote by $M_{2k}$ the disjoint union of $k$ disjoint copies of edges.

In 1959, Erd\H{o}s and Gallai \cite{erdHos1959maximal} characterized the extremal graphs for $M_{2k}$.
Later, Erd\H{o}s \cite{erdHos1962} studied the extremal graphs for $M^{3}_{2k}$ and Moon \cite{moon1968} determined the extremal graphs for $M^{p+1}_{2k}$ for infinitely many values of $n$ when $p\geq 3$.
After almost forty years, Erd\H{o}s, F\"{u}redi, Gould, and Gunderson \cite{erdHos1995} determined the Tur\'{a}n number of $S^{3}_{k+1}$ and Chen, Gould, Pfender, and Wei \cite{chen2003} determined the Tur\'{a}n number of $S^{p+1}_{k+1}$ for general $p\geq 3$.
Glebov \cite{glebov2011} determined the extremal graphs for edge blow-up of paths.
Later, Liu \cite{Liu2013} generalized Glebov's result to edge blow-up of paths, cycles and a class of trees.
Very recently, Wang, Hou, Liu, and Ma \cite{Wang2020} determined the Tur\'{a}n numbers for edge blow-up of a large family of trees.
For other extremal results concerning edge blow-up of specific graphs, we refer the interested readers to \cite{hou2016,ni2020,zhu2020}.
We will characterize the extremal graphs for edge blow-up of non-partite graphs and estimate  the Tur\'{a}n numbers of edge blow-up of bipartite graphs.
Our main theorems need some definitions, so we state them in Section~\ref{section:main theorems}.
As applications of our main theorems, see Sections~\ref{section:corollaries} and~\ref{section-proof-1.3}, we determine the Tur\'{a}n numbers of edge blow-up of complete bipartite graphs and complete graphs.

\begin{theorem}\label{coro6}
Let $K_t$ be the complete graph on $t$ vertices. For $p\geq t+1$ and sufficiently large $n$, we have
$$\emph{ex}(n,K_{t}^{p+1})={t-1 \choose 2}\left (n-{t-1 \choose 2}\right)+t_p\left(n-{t-1 \choose 2}\right).$$
Moreover, the extremal graphs  are characterized.
\end{theorem}

\begin{theorem}\label{coro5}
Let $K_{s,t}$ be the complete bipartite graph with $s\leq t$. For $p\geq 3$ and sufficiently large $n$, we have
$$\emph{ex}(n,K_{s,t}^{p+1})={s-1 \choose 2}+(s-1)(n-s+1)+t_p(n-s+1)+p(s,t),$$
where $p(s,t)$ is a constant depending on $s$ and $t$.
\end{theorem}

For a given graph $H$, let $g(n,H)$ denote the maximum number of edges not contained in any monochromatic copy of $H$ in a $2$-edge-coloring of $K_n$. If we color the edges of an extremal $n$-vertex graph for $H$ red and color the other edges blue, then we can see that $g(n,H)\geq \mbox{ex}(n,H)$ for any $H$ and $n$. In 2004, Keevash and Sudakov showed in  \cite{keevash2004} that this lower bound is tight for sufficiently large $n$ if $H$ is edge-critical or a cycle of length four. Hence, they posed the following conjecture.

\begin{conjecture}[Keevash and Sudakov \cite{keevash2004}]\label{Keevash-Sudakov-conjecture} Let $H$ be a given graph. If $n$ is sufficiently large, then
$$g(n,H)=\emph{ex}(n,H).$$
\end{conjecture}

Ma \cite{Ma-2017} and Liu-Pikhurko-Sharifzadeh \cite{Liu-Pikhurko-Sharifzadeh-2019} confirmed Conjecture~\ref{Keevash-Sudakov-conjecture} for a large family of bipartite graphs. Our method also works for this problem. In Section 6, we will show that Conjecture~\ref{Keevash-Sudakov-conjecture} does not hold for a large family of non-bipartite graphs. In particular, we prove the following theorem.

\begin{theorem}
Let $n$ be sufficiently large. Then
$$g(n,K_{t}^{p+1})=\emph{ex}(n,K_{t}^{p+1})+{{t-1\choose 2} \choose 2}.$$
\end{theorem}

The organisation of this paper is as follows.
In Section~\ref{section:main theorems}, we introduce some definitions and state our main theorems.
In Section~\ref{section:corollaries}, we present several corollaries of our main theorems.
In Section~\ref{section:lemmas}, we present several lemmas.
In Section~\ref{section5}, we will prove Theorems~\ref{maintheorem1} and~\ref{non-bipartite graph}.
In Section~\ref{section-proof-1.3}, we deduce some results about graphs without containing a matching with given sizes.
In Section~\ref{conclusion}, we will discuss more applications of our method.

\section{Main theorems}\label{section:main theorems}

Let $K_t$ be a complete  graph on $t$ vertices and $\overline{K}_t$ be the complement of $K_t$. Denote by $G\cup H$ the vertex-disjoint union of $G$ and $H$ and by $k\cdot G$ the vertex-disjoint union of $k$ copies of  $G$. Denote by $G+H$ the graph obtained from $G\cup H$ by adding edges between each vertex of $G$ and each vertex of $H$.

In order to study the Tur\'{a}n numbers of non-bipartite graphs, Simonovits \cite{Simonovits1974} defined the decomposition family $\mathcal{M}(\mathcal{F})$ of a family of graphs $\mathcal{F}$.

\begin{definition}[Simonovits \cite{Simonovits1974}]
\emph{Given a family of graphs $\mathcal{F}$ with $p(\mathcal{F})=p\geq 2$, let $\mathcal{M}(\mathcal{F})$ be the family of minimal\footnote{If $M\in \mathcal{M}(\mathcal{F})$, then $M^\prime\notin \mathcal{M}(\mathcal{F}$) where $M^\prime$ is a proper subgraph of $M$.} graphs $M$ satisfying the following: there exist an $F\in \mathcal{F}$ and a constant $t$ depending on $F$ such that $F\subseteq (M\cup\overline{K}_{t})+ T_{p-1}((p-1)t)$. We call $\mathcal{M}(\mathcal{F})$ the {\it decomposition family} of $\mathcal{F}$.}
\end{definition}

Thus, a graph $M$ is in $\mathcal{M}(\mathcal{F})$ if the graph obtained from putting a copy of $M$ (but not any of its proper subgraphs) into a class of a large $T_p(n)$ contains some $F\in \mathcal{F}$.
If $F\in\mathcal{ F}$ with chromatic number $p+1$, then $F\subseteq T_{p+1}((p+1)s)$ for some $s\geq 1$.
Therefore the decomposition family $\mathcal{M}(\mathcal{F})$ always contains some bipartite graphs\footnote{It is possible that the decomposition family contains non-bipartite graphs.}.
A deep theorem of Simonovits \cite{Simonovits1974} shows that if the decomposition family $\mathcal{M}(\mathcal{F})$  contains a linear forest\footnote{A linear forest is a graph consisting of paths.}, then the extremal graphs for $\mathcal{F}$ have very simple and symmetric structure (the theorem is quite complicated, we refer the interested readers to \cite{Simonovits1974} for more information).
Our theorems focus on graphs $F$ such that $\mathcal{M}(F)$ contains a matching (a matching is a linear forest).
Hence, our theorems are refinements of Simonovits' theorem in a certain sense. The main purpose of this paper is to determine the exact Tur\'{a}n numbers for new families of graphs.

To state our main theorems and related results, we need the following result.
Let $\Delta(G)$ be the maximum  degree of $G$. Define $f(\nu,\Delta)=\max\{e(G) :\nu(G)\leq \nu,\Delta(G)\leq \Delta\}$. In 1972, Abbott, Hanson, and Sauer \cite{abbott1972} determined $f(k-1,k-1)$. Later Chv\'{a}tal and Hanson \cite{chvatal1976} proved the following theorem.
\begin{theorem}[Chv\'{a}tal and Hanson \cite{chvatal1976}]\label{chvatal}
For every $\nu\geq 1$ and $\Delta\geq 1$,
\[f(\nu,\Delta)=\nu\Delta+\left\lfloor\frac{\Delta}{2}\right\rfloor\left\lfloor\frac{\nu}{\lceil\Delta/2\rceil}\right\rfloor\leq \nu\Delta+\nu.\]
\end{theorem}

In 2009, basing on Gallai's Lemma \cite{Gallai1963}, Balachandran and Khare \cite{Balachandran2009} gave a more `structural' proof of this result. Hence they gave a simple characterization of all the cases where the extremal graph is unique. Denote by $\mathcal{E}_{\nu,\Delta}$ the set of the extremal graphs in Theorem~\ref{chvatal}.

\begin{center}
\begin{tikzpicture}
\draw (1,0) ellipse (3 and 0.5);
\draw (1,3) ellipse (3 and 0.5);
\draw (1,5) ellipse (3 and 0.5);
\draw (5.5,2.5) ellipse (0.5 and 1.5);
\draw [line width=3pt](-1.5,0) -- (-1.4,3);
\draw [line width=3pt](-1.2,0) -- (-1.1,5);
\draw [line width=3pt](-0.9,3) -- (-0.85,5);
\filldraw[black] (1,1) circle (1pt);
\filldraw[black] (1,2) circle (1pt);
\filldraw[black] (1,1.5) circle (1pt);

\draw [line width=3pt](3,5) -- (5.2,3.5);
\draw [line width=3pt](3,3) -- (5.1,2.5);
\draw [line width=3pt](3,0) -- (5.2,1.5);
\draw node at (1.5,5){$E_{\nu,\Delta}$};
\draw node at (5.5,2.5){$Q_{s-1}$};
\draw node at (-3,2.5){$T_p(n-s+1)$};
\draw node at (1,-1){Figure 1: a graph of $\mathcal{H}(n,p,s,\nu,\Delta,\mathcal{B})$};
\end{tikzpicture}
\end{center}

Let $H(n,p,s)=K_{s-1}+T_p(n-s+1)$ and $H^\prime(n,p,s)=\overline{K}_{s-1}+T_p(n-s+1)$. Let $h(n,p,s)=e(H(n,p,s))$ and $h^\prime(n,p,s)=e(H^\prime(n,p,s))$. For a set of graphs $\mathcal{B}$, denote by $\mathcal{H}(n,p,s,\nu,\Delta,\mathcal{B})$ (Figure 1) the set of graphs which are obtained by taking an $H^\prime(n,p,s)$, putting a copy of $E_{\nu,\Delta}\in \mathcal{E}_{\nu,\Delta}$ in one class of $T_p(n-s+1)$ and putting a copy of $Q_{s-1}\in \mbox{EX}(s-1,\mathcal{B})$ in $\overline{K}_{s-1}$. As before, we use $\mathcal{H}(n,p,s,\nu,\Delta,B)$ to denote $\mathcal{H}(n,p,s,\nu,\Delta,\{B\})$ if $\mathcal{B}=\{B\}$.

A {\it covering} of a graph is a set of vertices which together meet all edges of the graph.
An {\it independent set} of a graph is a set of vertices no two of which are adjacent.
Similarly, an {\it independent covering} of a bipartite graph is an independent set which meets all edges of the bipartite graph. The minimum number of vertices in a
covering of a graph $F$ is called the {\it covering number} of $F$ and is denoted by $\beta(F)$.

Let $\mathcal{F}$ be a family of graphs containing at least one bipartite  graph. We need the following three parameters: $q(\mathcal{F})$, $\mathcal{S}(\mathcal{F})$, $\mathcal{B}(\mathcal{F})$, of $\mathcal{F}$ to describe our main theorems.

The {\it independent covering number} $q(\mathcal{F})$ of $\mathcal{F}$ is defined by
\[q(\mathcal{F})=\min\{q(F):F\in \mathcal{F} \mbox{ is bipartite}\},\]
where $q(F)$ is the minimum order of an  independent covering of $F$.

The {\it independent covering family} $\mathcal{S}(\mathcal{F})$ of $\mathcal{F}$ is the family of  independent coverings of  bipartite graphs $F\in \mathcal{F}$ with order $q(\mathcal{F})$.

The {\it subgraph covering family}  $\mathcal{B}(\mathcal{F})$ of $\mathcal{F}$ is the set of subgraphs (regardless of isolated vertices) of $F\in\mathcal{F}$ which are induced by a covering of $F$ with order at most $q(\mathcal{F})-1$ (if $\beta(F)\geq q(\mathcal{F})$ for each $F\in \mathcal{F}$, then we set $\mathcal{B}(\mathcal{F})=\{K_q\}$).

In the rest of this paper, given a graph $G$, let $G^{p+1}$ be the edge blow-up of $G$ with $p\geq \chi(G)+1$, $\mathcal{M}=\mathcal{M}(G^{p+1})$, $\mathcal{B}=\mathcal{B}(\mathcal{M})$ and $q=q(\mathcal{M})$. Let $k=\min\{d_{H_S}(x):x\in S, S\in \mathcal{S}(\mathcal{M})\}$, where $H_S\in \mathcal{M}$ contains $S$. For any connected bipartite graph $G$, let $A$ and $B$ be its two color classes with $|A|\leq |B|$. Moreover, if $G$ is disconnected, we always partition $G$ into $A\cup B$ such that $|A|$ is as small as possible. We will establish the following theorems.

\begin{theorem}\label{maintheorem1}
Let $G$ be a bipartite graph and $n$ be sufficiently large. For $p\geq 3$, we have the following:\\
(\romannumeral1). If  $q=|A|$, then
\begin{equation}\label{ineq in Main Theorem}
h^\prime(n,p,q)+\emph{ex}(q-1,\mathcal{B}) \leq \emph{ex}(n,G^{p+1})\leq h(n,p,q)+ f(k-1,k-1).
\end{equation}
Furthermore, both bounds are best possible.\\
(\romannumeral2). If $q<|A|$, then
$$\emph{ex}(n,G^{p+1})=h^\prime(n,p,q)+\emph{ex}(q-1,\mathcal{B}).$$
Moreover, the graphs in $\mathcal{H}(n,p,q,0,0,\mathcal{B})$ are the only extremal graphs for $G^{p+1}$.
\end{theorem}

\begin{theorem}\label{non-bipartite graph}
Let $G$ be a non-bipartite graph and $n$ be sufficiently large. For $p\geq \chi(G)+1$, we have
$$\emph{ex}(n,G^{p+1})=h^\prime(n,p,q)+\emph{ex}(q-1,\mathcal{B}).$$
Moreover, the graphs in $\mathcal{H}(n,p,q,0,0,\mathcal{B})$ are the only extremal graphs for $G^{p+1}$.
\end{theorem}

\section{Corollaries}\label{section:corollaries}

For a given graph $H$ with $\chi(H)=p+1\geq 3$, Erd\H{o}s-Stone-Simonovits theorem tells us that  the structure of the extremal graphs for $H$ are close to the Tur\'{a}n graph $T_p(n)$. More precisely, any extremal graph for $H$ can be obtained from  $T_p(n)$ by adding and deleting at most $o(n^2)$ edges. The decomposition family of a forbidden graph $H$ often helps us to determine the fine structure of the extremal graphs for $H$. Hence, we need the following lemmas concerning the extremal graphs of the decomposition family of $G^{p+1}$.

Given a graph $H$, {\it a vertex split} on some vertex $v\in V(H)$ is defined as follows: replace $v$ by an independent set of size $d(v)$ in which each vertex is adjacent to exactly one distinct vertex in $N_H(v)$. Denote by $\mathcal{H}(H)$ the family of graphs that can be obtained from $H$ by applying a vertex split on some $U\subseteq V(H)$. Obviously each graph in $\mathcal{H}(H)$ has $e(H)$ number of edges. Note that $U$ could be empty, therefore $H\in\mathcal{H}(H)$. For example, $\mathcal{H}(P_{k+1})$ is the family of all linear forests with $k$ edges and $\mathcal{H}(C_k)$ consists of $C_k$ and all linear forests with $k$ edges.

The following lemma is proved in \cite{Liu2013}.
\begin{lemma}[Liu \cite{Liu2013}]\label{liu lemma}
Given a graph $G$ with $2\leq \chi(G)\leq p-1$, we have $\mathcal{M}=\mathcal{H}(G)$, in particular, a matching of size $e(G)$ is in $\mathcal{M}$.
In particular, if $H\in \mathcal{M}$, then after splitting any vertex set of $H$, the resulting graph also belongs to $\mathcal{M}$.
\end{lemma}

Let $K_{n_1,\ldots,n_p}$ be the complete $p$-partite graph with class sizes $n_1,\ldots, n_p$. Denote by $K_{n_1,\ldots,n_p}(n,H_{n_1})$ the graph obtained by embedding $H_{n_1}$ into the class of $K_{n_1,\ldots,n_p}$ with size $n_1$.

\begin{proposition}\label{L-critical extremal 2}
Let  $F_{n_1}$ be an extremal graph for $\mathcal{M}$ on $n_1$ vertices. Then $K_{n_1,\ldots,n_p}(n,F_{n_1})$ does not contain $G^{p+1}$ as a subgraph.
\end{proposition}\begin{proof}
Proposition~\ref{L-critical extremal 2} follows directly from definition of decomposition family.\end{proof}

Theorem~\ref{maintheorem1} implies the results of Erd\H{o}s \cite{erdHos1962}, Moon, \cite{moon1968} and Simonovits \cite{Simonovits1968} for edge-blow up of matchings, results of Erd\H{o}s, F\"{u}redi, Gould, and Gunderson \cite{erdHos1995} and Chen, Gould, Pfender, and Wei \cite{chen2003} for edge-blow up of stars and results of Glebov \cite{glebov2011} and Liu \cite{Liu2013} for edge-blow up of paths and even cycles. Theorems~\ref{non-bipartite graph} implies the result of Liu \cite{Liu2013} for edge-blow up of odd cycles. We state those results as corollaries of Theorems~\ref{maintheorem1} and~\ref{non-bipartite graph}. In the following of this section, we will deduce the above results from our main theorems by applying Lemma~\ref{liu lemma}.

\begin{corollary}[Erd\H{o}s \cite{erdHos1962}, Moon \cite{moon1968} and Simonovits \cite{Simonovits1968}]\label{coro1}
Let $G=M_{2t}$ be a matching on $2t$ vertices and $p\geq2$. Then for sufficiently large $n$, we have
$$\emph{ex}(n,M_{2t}^{p+1})=h(n,p,t).$$
Moreover, $H(n,p,t)$ is the unique extremal graph for $M_{2t}^{p+1}$.
\end{corollary}
\begin{proof}
Clearly, we have $\mathcal{M}=\{M_{2t}\}$. Applying Theorem~\ref{maintheorem1} with $q=|A|=t$, $k=1$, $p\geq 3$, and $\mathcal{B}=\{K_q\}$, the lower and upper bounds of (\ref{ineq in Main Theorem}) are the same. Thus we have $\mbox{ex}(n,M_{2t}^{p+1})=h(n,p,t).$  The proof of Corollary~\ref{coro1} for $p\geq 3$ is complete. Since $\mathcal{M}$ contains only a matching $M_{2t}$, the proof of Theorem~\ref{maintheorem1} implies Corollary~\ref{coro1}  for $p=2$ (see Corollary~\ref{matching2}).
\end{proof}

\begin{corollary}[Erd\H{o}s, et al. \cite{erdHos1995} and Chen, et al. \cite{chen2003}]\label{coro2}
Let $G=S_{t+1}$ be a star on $t+1$ vertices and $p\geq 2$. Then, for sufficiently large $n$, we have
$$\emph{ex}(n,S_{t+1}^{p+1})=h(n,p,1)+f(t-1,t-1).$$
Moreover, the extremal graphs are characterized.
\end{corollary}
\begin{proof}
By Lemma~\ref{liu lemma}, we have $\mathcal{M}=\{M_{2t},S_{t+1}\}$.
Note that the graphs in $\mathcal{H}(n,p,1,k-1,k-1,K_2)$ does not contain $S_{t+1}^{p+1}$ as a subgraph (by Proposition~\ref{L-critical extremal 2}).
Applying Theorem~\ref{maintheorem1}(\romannumeral1) with $q=|A|=1$, $k=t$, $p\geq 3$, and $\mathcal{B}=\{K_1\}$, we have ex$(n,S_{t+1}^{p+1})\leq h(n,p,1)+f(t-1,t-1)$.
The proof of Corollary~\ref{coro2} for $p\geq 3$ is complete.
Note that $\mathcal{M}$ contains a matching.
The proof of Theorem~\ref{maintheorem1} implies\footnote{We omit the proof, since it is essentially the same as the proof of Theorem~\ref{maintheorem1}.} Corollary~\ref{coro2} for $p=2$.
\end{proof}

\begin{corollary}[Glebov \cite{glebov2011} and Liu \cite{Liu2013}]\label{coro3}
Let $G=P_{t}$ be a path on $t$ vertices and $p\geq 3$. Then, for sufficiently large $n$, we have
$$\emph{ex}(n,P_{t}^{p+1})=h\left(n,p,\left\lfloor\frac{t}{2}\right\rfloor\right)+i,$$
where $i=1$ when $t$ is odd and $i=0$ when $t$ is even.
\end{corollary}
\begin{proof}
By Lemma~\ref{liu lemma}, $\mathcal{M}$ consists of all linear forests with $t-1$ edges.
For a linear forest $F$ in $\mathcal{M}$ consisting of  paths $P_{t_1}$, $P_{t_2}$, $\ldots$, $P_{t_\ell}$,  each covering of $F$ has at least $\sum_{i=1}^{\ell}\lfloor t_i/2\rfloor\geq (t-1)/2$ vertices.
If $t$ is even, then $k=1$. Since $(q-1)S_3\cup S_2\in \mathcal{M}$ and the minimum non-independent coverings of linear forests with $t-1$ edges is $t/2+1$, applying Theorem~\ref{maintheorem1}(\romannumeral1) with $q=|A|= t/2$, $k=1$, $p\geq 3$, and $\mathcal{B}=\{K_q\}$, the lower and upper bounds of (\ref{ineq in Main Theorem}) are the same. Assume that $t$ is odd. Then $k=2$. Note that the graphs in $\mathcal{H}(n,p,q,1,1,K_q)$ do not contain a copy of $P_{t}^{p+1}$, where $q=|A|=\lfloor t/2\rfloor$. It follows from the upper bound of (\ref{ineq in Main Theorem}) that $\mbox{ex}(n,P_{t}^{p+1})=h(n,p,q)+1$. Thus, the proof of Corollary~\ref{coro3} is complete.
\end{proof}

\begin{corollary}[Liu \cite{Liu2013}]\label{coro4}
Let $G=C_{t}$ be a cycle on $t$ vertices. Then, for sufficiently large $n$, we have the following:\\
(a) If $t$ is even and $p\geq 3$, then
$$\emph{ex}(n,C_{t}^{p+1})=h\left(n,p,\left\lfloor\frac{t}{2}\right\rfloor\right)+1.$$
(b) If $t$ is odd and $p\geq 4$, then
$$\emph{ex}(n,C_{t}^{p+1})=h\left(n,p,\left\lceil\frac{t}{2}\right\rceil\right).$$
\end{corollary}
\begin{proof}
By Lemma~\ref{liu lemma}, $\mathcal{M}$ consists of all linear forests with $t$ edges and the cycle of length $t$. Let $t$ be even. Since the graphs in  $\mathcal{H}(n,p,t/2,1,1,K_{t/2})$ do not contain $C_{t}^{p+1}$ as a subgraph (by Proposition~\ref{L-critical extremal 2}), applying Theorem~\ref{maintheorem1}(\romannumeral1) with $q=|A|=t/2$, $k=2$, $p\geq  \chi(C_t)+1 =3$ and $\mathcal{B}=\{K_q\}$, the lower bound and the upper bound of  Theorem~\ref{maintheorem1} are the same. The proof of Corollary~\ref{coro4}(a) is complete. Let $t$ be odd. Then Corollary~\ref{coro4}(b) follows from Theorem~\ref{non-bipartite graph} with $q=\lceil t/2\rceil$, $\mathcal{B}=\{K_q\}$,  and $p\geq \chi(C_t)+1 =4$.\end{proof}

\noindent {\bf Proof of Theorem~\ref{coro6}}:\\
\begin{proof} Let $G=K_t$. Denote by $S_{k,k}$ the graph on $2k$ vertices obtained by taking two copies of  $S_{k}$ and joining the centers of them with a new edge.
Since each bipartite graph in $\mathcal{M}$  is obtained by splitting at least $t-2$ vertices of $K_t $, we have $q=t-1+{t-2 \choose 2}={t-1 \choose 2}+1$ (the graph  $F\in\mathcal{M}(K_{t}^{p+1})$ consisting of $S_{t-1,t-1}$ and ${t-2 \choose 2}$ independent edges) and $\mathcal{B}=\{K_2\}$ (the edge joining the centers of $S_{t-1}$ in $S_{t-1,t-1}$ of $F$). Applying Theorem~\ref{non-bipartite graph}, we have $\mbox{ex}(n,K_{t}^{p+1})={t-1 \choose 2}\left (n-{t-1 \choose 2}\right)+t_p\left(n-{t-1 \choose 2}\right).$
Moreover, the extremal graphs  are characterized. The proof of Theorem~\ref{coro6} is complete.\end{proof}

The proof of Theorem~\ref{coro5} needs more efforts, so we move it to Section~\ref{section-proof-1.3}.

\section{Several technical lemmas}\label{section:lemmas}

The following simple propositions help us to determine the extremal graphs for $\mathcal{M}$.

\begin{proposition}\label{3.8}
Let $F$ be a bipartite graph.  Then we have $q(F)=|A|.$
\end{proposition}
\begin{proof} Since $A$ is an independent covering of $F$, we have $q(F)\leq|A|$.
Suppose $F$ is connected. Then each independent covering of $F$ must contain either all the vertices of $A$ or all the vertices of $B$.
Indeed, assume that $A_1\subsetneq A$, $B_1\subsetneq B$ are two non-empty vertex sets and $A_1\cup B_1$ is an independent covering of $F$.
Let $A_2=A-A_1$ and $B_2=B-B_1$.
Since $F$ is connected and $A_1\cup B_1$ is an independent set, there is some edge between $A_2$ and $B_2$, contradicting that $A_1\cup B_1$ is a  covering of $F$.
Hence we have $q(F)=|A|$.
If $F$ is disconnected, the result follows easily by studying each component of $F$ (recall that we always partition $F$ with $|A|$ as small as possible).
The proof is complete.\end{proof}

We need the following proposition to determine the extremal graphs for $G^{p+1}$ when $k=1$ (recall definitions of $\mathcal{S}(\mathcal{M})$, $q$ and $k$).

\begin{proposition}\label{3.9}
If there is an independent covering $S\in\mathcal{S}(\mathcal{M})$ obtained by splitting some vertices in $G$, then $k=1$.
Moreover, if $G$ is bipartite with $q<|A|$ or $G$ is non-bipartite, then $k=1$.
\end{proposition}
\begin{proof} Let $H_S\in \mathcal{M}$ be a bipartite graph with $q(H_S)=q$ and $S$ be an independent covering of $H_S$ with order $q$.
Since each  vertex in $S$ obtained by splitting a vertex in $G$ has degree one in $H_S$, by definition of $k$, we have $k=1$.
Let $G$ be a bipartite graph with $q<|A|$.
Then there is an $x\in S$ which is obtained by splitting a vertex in $G$.
Otherwise, by Proposition~\ref{3.8}, we have $q=|A|$, a contradiction.
Thus, we have $k=1$.
Let  $G$ be a non-bipartite graph. Then, there is an $x\in S$ which is obtained by splitting a vertex in $G$.
Otherwise, $G$ has an independent covering and hence is bipartite, a contradiction.
The result follows similarly as before.
\end{proof}

Now, we will study the Tur\'{a}n number  of the decomposition family of $G^{p+1}$ which helps us to determine  the Tur\'{a}n number of $G^{p+1}$.

\begin{lemma}\label{main1}
Suppose that $n$ is sufficiently large. Then we have the following.\\
(a). If $G$ is bipartite with $q=|A|$, then
\begin{equation}\label{4}
h^\prime(n,1,q)+\emph{ex}(q-1,\mathcal{B})\leq \emph{ex}(n,\mathcal{M})\leq h(n,1,q)+ f(k-1,k-1).
\end{equation}
Furthermore, both bounds are best possible.\\
(b). If $G$ is bipartite with $q<|A|$ or $G$ is non-bipartite, then
\begin{equation}\label{4.11}
 \emph{ex}(n,\mathcal{M})=h^\prime(n,1,q)+\emph{ex}(q-1,\mathcal{B}).
\end{equation}
Moreover, the extremal graphs are characterized.
\end{lemma}
\begin{proof}
Let $H\in \mathcal{M}$ be a bipartite graph with an independent covering $S\in \mathcal{S}(\mathcal{M})$ and a vertex $x\in S$ such that $d_H(x)=k$. Let $G^{\prime}$ be an extremal graph for $\mathcal{M}$.\\
\noindent$(a)$. Assume that $G$ is bipartite and $q=|A|$. For the upper bound of (\ref{4}), suppose that
\begin{equation}\label{4.1}
e(G^{\prime})\geq h(n,1,q)+ f(k-1,k-1)={q-1 \choose 2}+(q-1)(n-q+1)+f(k-1,k-1).
\end{equation}
First, there are at most $q-1$ vertices of $G^{\prime}$ with degree at least $e(G)+q$.
Otherwise, by Lemma~\ref{liu lemma}, $G^{\prime}$ contains a copy of $H_1\in \mathcal{M}$ ($H_1$ is a star forest\footnote{A star forest is a graph consisting of stars.})  obtained from $H$ by splitting all vertices of $H-S$, a contradiction.
Suppose that the number of  vertices of $G^{\prime}$ with degree at least $e(G)+q$ is less than $q-1$.
By Lemma~\ref{liu lemma}, $\mathcal{M}$ contains a matching with size $e(G)$.
Since $n$ is sufficiently large and $f(\Delta,\nu)$ is a constant depending on $\Delta$ and $\nu$, we have
\begin{align*}
e(G^\prime)&\leq (q-2)(n-1)+f(e(G)+q,e(G))\\
&<{q-1 \choose 2}+(q-1)(n-q+1)+f(k-1,k-1),
\end{align*}
contradicting (\ref{4.1}). Thus, there are exactly $q-1$ vertices of $G^\prime$ with degree at least $e(G)+q$.
Let $X$ be set of vertices with degree at least $e(G)$ and $\widetilde{G}=G^{\prime}-X$.
Then $\widetilde{G}$ contains neither $S_{k+1}$ nor $M_{2k}$ as a subgraph.
Otherwise, by Lemma~\ref{liu lemma}, $G^\prime$ contains either a copy of $H_1\in \mathcal{M}$ or a copy of $H_2\in \mathcal{M}$ obtained from $H$ by splitting  all vertices of $H-S$ and the vertex $x$.
Hence, we have $e(\widetilde{G})\leq f(k-1,k-1)$.
Then by (\ref{4.1}), we have that $e(\widetilde{G})=f(k-1,k-1)$ and each vertex in $X$ has degree $n-1$.
Moreover, it follows from Theorem~\ref{chvatal} that $\widetilde{G}\in \mathcal{E}_{k-1,k-1}$.
Thus we have  $G^\prime\in \mathcal{H}(n,1,q,k-1,k-1,K_{q})$ or $e(G^\prime)< h(n,1,q)+ f(k-1,k-1)$.
The lower bound of (\ref{4}) follows from the fact  that the graphs in $\mathcal{H}(n,1,q,0,0,\mathcal{B})$ do not contain any graph in $\mathcal{M}$ as a subgraph (by definitions of $q$ and $\mathcal{B}$).

\noindent$(b)$. Now let $G$ be a bipartite graph\footnote{The graph $G=F_{t,t}$ obtained from by taking two copies of $S_t$ with $t\geq 4$ and joining two leaves in deferent $S_t$'s satisfies that $q=3<t=|A|$ (splitting the vertices of the added edge).} with $q<|A|$ or be a non-partite graph.
Thus, it follows from the definitions of $q$ and $\mathcal{B}$ that each graph in $\mathcal{H}(n,1,q,0,0,\mathcal{B})$ does not contain any graph in $\mathcal{M}$ as a subgraph.
Thus we have $e(G^\prime)\geq h^\prime(n,1,q)+\mbox{ex}(q-1,\mathcal{B})$. Let $X$ be the set vertices of $G^\prime$ with degree at least $e(G)$ and $\widetilde{G}=G^\prime-X$.
Similarly as the previous arguments, we have $|X|=q-1$. It follows from Proposition~\ref{3.9} that $k=1$.
Hence, we have $e(\widetilde{G})=0$. So we have $e(G^\prime)= h^\prime(n,1,q)+\mbox{ex}(q-1,\mathcal{B})$.
Moreover, the extremal graphs are in $\mathcal{H}(n,1,q,0,0,\mathcal{B})$. \end{proof}

\begin{lemma}\label{lemma for F(k,p)(1)}
Let $F$ be a graph with a partition of vertices into $p+1$ parts $V(F)=V_0\cup V_1\cup V_2\cup \ldots \cup V_p$ satisfying the following:\\
\indent(1) there exist $V^{\prime}_1\subseteq V_1$, $\ldots$, $V^{\prime}_p\subseteq V_p$ such that $F[V^{\prime}_1\cup \ldots \cup V^{\prime}_p]=T_p(ap)$;\\
\indent(2) $|V_0|=q-1$ and each vertex of $V_0$ is adjacent to each vertex of $T_p(ap)$;\\
\indent(3) each vertex of $V^{\prime\prime}_i=V_i\setminus V^{\prime}_i$ is adjacent to each vertex of $V^{\prime}_{j\neq i}$ for $i\in[p]$.\\
Let $G$ be a bipartite graph with $q=|A|$ and $|V(G)|\leq a$. If there exist an $i\in[p]$ and a $y \in V^{\prime\prime}_i$ such that one of the following holds:\\
\indent(a) $\sum_{j\neq i}\nu(F[V^{\prime\prime}_j])\geq k$;\\
\indent(b) $\Delta(F[V^{\prime\prime}_i])\geq k$;\\
\indent(c) $d_{F[V^{\prime\prime}_{i}]}(y)+\sum_{j\neq i}\nu(F[N(y)\cap V^{\prime\prime}_j])\geq k$,\\
then $F$ contains a copy of $G^{p+1}$.
\end{lemma}
\begin{proof}
If $k=1$, then the lemma holds trivially by definition of $k$ and definition of decomposition family (the graph $H$ consisting of $q-1$ stars and an isolated edges belongs to $\mathcal{M}$ and $F[V_0\cup V_{i^\prime}]$ contains a copy of $H$, where $F[V_{i^\prime}]$ contains an edge).
Assume that $k\geq 2$, i.e., there is no independent covering $S\in\mathcal{S}(\mathcal{M})$ obtained by splitting some vertices in $G$.
Then it follows from Proposition~\ref{3.8} that $q=|A|$ and hence $k=\min\{d_G(x):x\in A \mbox{ or }x\in G \mbox{ when }|A|=|B|\}$.
Let $V^{\prime}_i=\{x_{i,1},x_{i,2}\ldots,x_{i,a}\}$ for $i\in[p]$ and $F^{\prime}=F-V_0$.
Since each vertex of $V_0$ is adjacent to each vertex of $V^\prime_i$, $|V_0|=q-1=|A|-1$, and $a\geq |V(G)|$, it is enough to show that $F^{\prime}$ contains a copy of $S^{p+1}_{k+1}$ with the following property: each copy of $K_p$ in $S^{p+1}_{k+1}$ without containing the center\footnote{The center of $S^{p+1}_{k+1}$ is the vertex in $S^{p+1}_{k+1}$ with degree $pk$.} of $S^{p+1}_{k+1}$ contains at least one vertex in $\cup_{i=1}^{p}V^{\prime}_i$ .
In fact, we map the center of $S^{p+1}_{k+1}$ and $V_0$ to either $A$ of $G^{p+1}$, or $B$ of $G^{p+1}$ when  there is a vertex $x\in B$ with degree $k$ with $|A|=|B|$.
We will prove the lemma in the following three cases.

\medskip

\noindent{\bf Case 1.} $\sum_{j\neq i}\nu(F[V^{\prime\prime}_j])\geq k$. Without loss of generality, let $\sum_{j\neq 1}\nu(F[V^{\prime\prime}_j])\geq k$. Let $\{y_1z_1,y_2z_2,\ldots,y_kz_k\}$ be a  matching in $\cup_{j\neq 1}F[V^{\prime\prime}_j]$ and
$$F_s=F[x_{1,1},y_s,z_s,x_{2,s},x_{3,s},\ldots,x_{p,s}]$$
for $s\in [k]$. Clearly, we have $F_s=K_{p+1}$ for $s\in [k]$ and $V(F_s)\cap V(F_t) =\{x_{1,1}\}$ for $s\neq t$. Since $x_{2,s}\in \cup_{i=1}^{p}V^{\prime}_i$ for $s\in [k]$,  we obtain the desired copy of  $S^{p+1}_{k+1}$, the result follows.

\medskip

\noindent{\bf Case 2.} $\Delta(F[V^{\prime\prime}_i])\geq k$. Without loss of generality, let $\Delta(F[V^{\prime\prime}_1])\geq k$, $y$ be a vertex in $V^{\prime\prime}_1$ with $d_{F[V^{\prime\prime}_1]}(y)\geq k$ and $x_1,x_2,\ldots,x_k$ be the neighbours of $y$ in $V^{\prime\prime}_1$. Let
$$F_s=F[y,x_s,x_{2,s},x_{3,s},\ldots,x_{p,s}]$$
for $s \in [k]$. Clearly we have $F_s=K_{p+1}$ for $s\in [k]$ and $V(F_s)\cap V(F_t) =\{y\}$ for $s\neq t$. Since $x_{2,s}\in \cup_{i=1}^{p}V^{\prime}_i$ for $s\in [k]$,  we obtain the desired copy of  $S^{p+1}_{k+1}$, the result follows.

\medskip

\noindent{\bf Case 3.} $d_{F[V^{\prime\prime}_{i}]}(y)+\sum_{j\neq i}\nu(F[N(y)\cap V^{\prime\prime}_j])\geq k$. Without loss of generality, let $d_{F[V^{\prime\prime}_{1}]}(y)+\sum_{j\neq 1}\nu(F[N(y)\cap V^{\prime\prime}_j])\geq k$. Let $d_{F[V^{\prime\prime}_{1}]}(y)=t<k$, $x_1,x_2,\ldots,x_t$ be the neighbours of $y$ in $F[V^{\prime\prime}_{1}]$ and  $\{y_{t+1}z_{t+1},\ldots,y_kz_k\}$ be a matching in $\bigcup_{j\neq 1}F[N(y)\cap V^{\prime\prime}_j]$. Let
$$F_s=\left\{\begin{array}{ll}\
F[y,x_s,x_{2,s},x_{3,s},\ldots,x_{p,s}] &\mbox{ for } s=1,2,\ldots,t,\\
F[y,y_s,z_s,x_{2,s},x_{3,s},\ldots,x_{p,s}] &\mbox{ for } s=t+1,t+2,\ldots,k.
\end{array}\right.$$
Clearly, we have $F_s=K_{p+1}$ for $s\in [k]$ and $V(F_s)\cap V(F_t) =\{y\}$ for $s\neq t$. Since $x_{2,s}\in \cup_{i=1}^{p}V^{\prime}_i$ for $s\in [k]$,  we obtain the desired copy of  $S^{p+1}_{k+1}$, the result follows.
\end{proof}

Let $G$ be a graph with a partition of the vertices into $p\geq 2$ non-empty parts
\[V(G)=V_1\cup V_2\cup \ldots \cup V_p.\]
Let $G_i=G[V_i]$ for $i=1,2,\ldots,p$ and define
\[G_{cr}=(V(G),\{v_iv_j:v_i\in V_i,v_j\in V_j,i\neq j\}),\]
where ``cr'' denotes ``crossing''. The following lemma is proved in \cite{chen2003}.

\begin{lemma}[Chen, Gould, Pfender, and Wei \cite{chen2003}]\label{lemma for F(k,p)}
Let $G$ be a graph on $n$ vertices. Suppose $G$ is partitioned as above so that
\begin{eqnarray}
\sum_{j\neq i}\nu(G[V_j])\leq k-1  \mbox{ and } \Delta(G[V_i])\leq k-1;\label{11}\\
d_{G[V_{i}]}(x)+\sum_{j\neq i}\nu(G[N(x)\cap V_j])\leq k-1.\label{12}
\end{eqnarray}
are satisfied for all $i$ and for all $x\in V_i$. If $G$ does not contain a copy of $S_{k+1}^{p+1}$, then
\begin{eqnarray}\label{equation for f(k,r)}
\sum_{i=1}^{p}|E(G_i)|-\left(\sum_{1\leq i<j\leq p}|V_i||V_j|-|E(G_{cr})|\right)\leq f(k-1,k-1).
\end{eqnarray}
Moreover, if the equality holds, then
\begin{eqnarray}\label{equation for f(k,r)(2)}
 \sum_{1\leq i<j\leq p}|V_i||V_j|=|E(G_{cr})|,\; e(G[V_i])=f(k-1,k-1),\; e(G[V_{\ell\neq i}])=0,
 \end{eqnarray}
  and $G[V_i]\in \mathcal{E}_{k-1,k-1}$ for some $i\in\{1,\ldots,p\}$.
Furthermore, if $\sum_{j\neq i}|E(G_j)|\geq 1$ for each $i\in\{1,\ldots,p\}$, i.e., at least two of $E(G_1),\ldots,E(G_p)$ are non-empty, then
\begin{eqnarray}\label{equation for K_s,t}
\sum_{i=1}^{p}|E(G_i)|-\left(\sum_{1\leq i<j\leq p}|V_i||V_j|-|E(G_{cr})|\right)\leq k^2-2k.
\end{eqnarray}

\end{lemma}
\remark By analyzing the proof of Lemma~\ref{lemma for F(k,p)}, it is not difficult to see that if the equality holds in (\ref{equation for f(k,r)}), then (\ref{equation for f(k,r)(2)}) is satisfied and $G[V_i]\in \mathcal{E}_{k-1,k-1}$ (See Lemma 2.7 in \cite{Yuan2018}). The proof of  Lemma~\ref{lemma for F(k,p)} also implies the last assertion of Lemma~\ref{lemma for F(k,p)} (See the last sentence of the proof of Lemma 3.2 in \cite{chen2003}).

\medskip

In 1968, Simonovits \cite{Simonovits1968} introduced the so-called {\it progressive induction} which is similar to the mathematical induction and Euclidean algorithm and combined from them in a certain sense.
\begin{lemma}[Simonovits \cite{Simonovits1968}]\label{progrssion induction}
Let $\mathfrak{U}=\cup_{i=1}^{\infty}\mathfrak{U}_i$ be a set of given elements, such that $\mathfrak{U}_i$ are disjoint finite subsets of $\mathfrak{U}$. Let $B$ be a condition or property defined on $\mathfrak{U}$ (i.e. the elements of $\mathfrak{U}$ may satisfy or not satisfy $B$). Let $\phi(a)$ be a function defined  on $\mathfrak{U}$ such that $\phi(a)$ is a non-negative integer and\\
(a) if $a$ satisfies $B$, then $\phi(a)=0$.\\
(b) there is an $M_0$ such that if $n>M_0$ and $a\in \mathfrak{U}_n$ then either $a$ satisfies $B$ or there exist an $n^{\prime}$ and an $a^{\prime}$ such that
\[\frac{n}{2}<n^{\prime}<n, a^{\prime}\in \mathfrak{U}_{n^{\prime}} \mbox{ and } \phi(a)<\phi(a^{\prime}).\]
Then there exists an $n_0$ such that if $n>n_0$, every $a\in \mathfrak{U}_n$  satisfies $B$.
\end{lemma}

\remark In our problems, $\mathfrak{U}_n$ is a set of extremal graphs for $G^{p+1}$ on $n$ vertices, $B$ is the property defined on $\mathfrak{U}$ concerning  the number of edges or the structure of graphs.
\medskip

\section{Proof of the main theorems}\label{section5}

\noindent {\bf Proof of Theorem~\ref{maintheorem1} (\romannumeral1):}\\
\begin{proof}
We will prove that, for sufficiently large $n$,
\begin{equation}\label{eq-for-proof-1}
h^\prime(n,p,q)+\mbox{ex}(q-1,\mathcal{B})  \leq \mbox{ex}(n,G^{p+1})\leq h(n,p,q)+ f(k-1,k-1)
\end{equation}
and if $\mbox{ex}(n,G^{p+1})= h(n,p,q)+ f(k-1,k-1)$, then $\mbox{EX}(n,G^{p+1})=\mathcal{H}(n,p,q,k-1,k-1,K_{q})$.
Lemma~\ref{main1} together with Proposition~\ref{L-critical extremal 2} implies the lower bound of (\ref{eq-for-proof-1}). We will prove the upper bound of (\ref{eq-for-proof-1}) by Lemma~\ref{progrssion induction}. Suppose $F_n$ is an extremal graph for $G^{p+1}$. It will be shown that, if $n$ is sufficiently large, then $e(F_n)\leq h(n,p,q)+f(k-1,k-1).$ Let $H_n\in \mathcal{H}(n,p,q,k-1,k-1,K_{q})$. Clearly, $e(H_n)=h(n,p,q)+ f(k-1,k-1)$. If $e(F_n)< e(H_n)$, then we are done. Let
\begin{equation}\label{lemma 1}
e(F_n)\geq e(H_n).
 \end{equation}
Let $\mathfrak{U}_n$ be the set of extremal graphs for $G^{p+1}$ on $n$ vertices and $B$ be the property defined on $\mathfrak{U}$ stating that $e(F_n)\leq e(H_n)$ and equality holds if and only if $F_n \in \mathcal{H}(n,p,q,k-1,k-1,K_{q})$. Define $\phi(F_n)=\max\{e(F_n)-e(H_n),0\}$. Then $\phi(F_n)$ is a non-negative integer. According to Lemma~\ref{progrssion induction}, it is enough to show that if $e(F_n)\geq e(H_n)$, then either $F_n\in \mathcal{H}(n,p,q,k-1,k-1,K_{q})$ or there exists an $n^{\prime}\in (n/2,n)$ such that $\phi(F_{n^{\prime}})>\phi(F_n)$ when $n$  is sufficiently large.

Now, we will find a subgraph of $F_n$ satisfying the conditions of Lemma~\ref{lemma for F(k,p)}.
Since  $e(H_n)\geq t_p(n)$, by Theorem~\ref{erdos-stone} and (\ref{lemma 1}), there is an $n_1$ such that if $n>n_1$, then  $F_n$ contains $T_p(n_2p)$ ($n_2$ is sufficiently large) as a subgraph. Since $2\leq \chi(G)\leq p-1$,  it follows from Lemma~\ref{liu lemma} that $\mathcal{M}$ contains a matching $M_{2k_1}$, where $k_1=e(G)$. Each partite class of $T_p(n_2p)$ contains no copy of $M_{2k_1}$. Otherwise, it follows from definition of decomposition family that $F_n$ contains a copy of $G^{p+1}$, a contradiction. Let $\{x_1y_1,x_2y_2,\ldots, x_{s_1}y_{s_1}\}$ be a maximum matching in one class, say $B^\prime_1$, of $T_p(n_2p)$ and let $\widetilde{B}_1=B^\prime_1-\{x_1,y_1,\ldots, x_{s_1},y_{s_1}\}$. Then there is no edge in $F_n[\widetilde{B}_1]$. Hence there is an induced subgraph $T_p(n_3p)$ ($n_3=n_2-2k_1$ is sufficiently large) of $F_n$ with partite sets $B_1,B_2,\ldots,B_p$ obtained by deleting $2k_1$ vertices of each class of $T_p(n_2p)$.

Let $c$ be a sufficiently small constant and $S=V(F_n)\setminus V(T_p(n_3p))$.
Let $T_0=T_p(n_3p)$. For $i \geq 1$, we will define vertices $x_i \in S$ and graphs $T_i$ recursively.
If there is a $u \in S\setminus \{x_1,\ldots,x_{i-1}\}$ which has at least $c^{2i}n_{3}$ neighbors in each class of $T_{i-1}$, then let $x_i=u$ and define $T_i$ as a subgraph of $T_{i-1}$ which is isomorphic to $T_p(c^{2i}n_3p)$.
By the definition, we get that $\{x_1,\ldots,x_i\}$ together with $T_i$ form a complete $(p+1)$-partite graph. We claim $i \leq q-1$. Otherwise,
 we easily find a copy of $T_{q}\subseteq F_n$ with $V(T_q)=B_1^q\cup \ldots \cup B_p^q$. Thus the induced subgraph of $F_n$ on $B^q_1\cup \{x_1,x_2,\ldots,x_{q}\}$ is a graph with $q+c^{2q}n_3$ vertices and at least $c^{2q}n_3q$ edges.
Since
\[
c^{2q}n_3q>{q-1 \choose 2}+(q-1)(c^{2q}n_3+1)+f(k-1,k-1),
\]
 by Lemma~\ref{main1}, this induced subgraph contains a copy of  $M\in \mathcal{M}$ provided  $n_3$ is large enough. Note that each vertex of this induced subgraph is adjacent to each vertex of $B^q_2\cup \ldots \cup B^q_p$. By definition of decomposition family, $S_n$ contains a copy of $G^{p+1}$, a contradiction.

Now, suppose the above process ends at $T_\ell$ with $0\leq \ell\leq q-1$.
Denote by $x_1,x_2,\ldots,x_\ell$ the vertices joining to all the vertices of $T_\ell$.
Let $B^{\ell}_1,\ldots,B^\ell_p$ be the classes of $T_\ell$.
Partition the remaining vertices into the following vertex sets:
Let $x\in V(F_n)\setminus (V(T_\ell)\cup \{x_1,\ldots,x_{\ell}\})$.
If there exists $i\in[p]$ such that $x$ is adjacent to less than $c^{2\ell+2}n_3$ vertices of $B^\ell_i$ and is adjacent to at least  $(1-c)c^{2\ell}n_3$ vertices of $B^\ell_j$ for $j\neq i$, then let $x\in C^\ell_i$.
If there exists $i\in[p]$ such that $x$ is adjacent to less than $c^{2\ell+2}n_3$ vertices of $B^\ell_i$ and is adjacent to less than $(1-c)c^{2\ell}n_3$ vertices of some of $B^\ell_j$ with $j\neq i$, then let $x\in D$.
Obviously, $$C^\ell_1\cup\ldots \cup C^\ell_p\cup D$$ is a partition of $V(F_n)\setminus (V(T_\ell)\cup \{x_1,\ldots,x_{\ell}\})$.
Since $\mathcal{M}$ contains a matching with size $k_1$ and each vertex of $C^\ell_i$ is adjacent to less than $c^{2\ell+2}n_3$ vertices of $B^\ell_i$, there are at least $c^{2\ell}n_3(1-c^2k_1)$ vertices of $B^\ell_i$ which are not adjacent to any vertices of $C^\ell_i$.
Indeed, there are at most $k_1$ independent edges in $B^\ell_i\cup C^\ell_i$.
Otherwise, since each vertex in $C^\ell_i$ is adjacent to  at least $(1-c)c^{2\ell}n_3$ of $B^\ell_{j\neq i}$, it is easy to see that $F_n$ contains a copy of $G^{p+1}$, a contradiction.
Consider the edges joining $B^\ell_i$ and $C^\ell_i$ and select a maximal set of independent edges, say $x_1y_1,\ldots,x_ty_t$, $x_{i^{\prime}}\in B^\ell_i$, $y_{i^{\prime}}\in C^\ell_i$, $1\leq i^{\prime}\leq t\leq k_1$, among them, then the number of vertices of $B^\ell_i$ joining to at least one of $y_1,y_2,\ldots,y_t$ is less than $c^{2\ell+2}n_3q$, and the remaining vertices of $B^\ell_i$ are not adjacent to any vertices of $C_i$ by the maximality of $x_1y_1,\ldots,x_ty_t$.
Hence we can move $c^{2\ell+2}n_3k_1$ vertices of $B^\ell_i$  to $C^\ell_i$, obtain $B^\prime_i$ and $C^\prime_i$ such that $B^\prime_i\subseteq B^\ell_i$, $C^\ell_i\subseteq C^\prime_i$, and there is no edge between $B^\prime_i$ and $C^\prime_i$.
Let $n_4=(1-c^2k_1)c^{2\ell}n_3$.
We conclude that $T^{\prime}_\ell=T_p(n_4 p)$ with classes $B^\prime_1,\ldots,B^\prime_p$ is an induced subgraph of $F_n$ satisfying the following conditions:

Let $\widehat{F}=F_n-T^{\prime}_\ell$. The vertices of $\widehat{F}$ can be partitioned into $p+2$ classes $C^\prime_1,\ldots,C^\prime_p$, $D$ and $E$ such that
\begin{itemize}
\item each $x\in E$ is adjacent to each vertex of $T^{\prime}_\ell$ and $|E|=\ell$;
\item if $x\in C^\prime_i $ then $x$ is adjacent to at least $(1-c-c^2k_1)c^{2\ell}n_3$ vertices of $B^\prime_{j\neq i}$ and is adjacent to no vertex of $B^\prime_i$;
\item if $x\in D$, then there are two different classes, $B^\prime_{i}$ and $B^\prime_{j}$, of $T^{\prime}_\ell$ such that $x$ is adjacent to less than $c^{2\ell+2}n_3$ vertices of $B^\prime_{i}$  and less than $(1-c)c^{2\ell}n_3$ vertices of $B^\prime_{j}$.
\end{itemize}
Denote by $e_F$ the number of the edges joining $\widehat{F}$ and $T^\prime_\ell$. Clearly, we have
\begin{equation}\label{lemma 2}
e(F_n)=e(T^\prime_\ell)+e_F+e(\widehat{F}).
\end{equation}
Select an induced copy of $T^{\prime}_\ell$ in $H_n$, let $H_{n-n_4 p} =H_n-T^{\prime}_\ell$ and $e_T$ be the number of edges of $H_n$ joining $T^{\prime}_\ell$ and $H_{n-n_4 p}.$ Then, we have
\begin{equation}\label{lemma 3}
e(H_n)=e(T^\prime_\ell)+e_T+e(H_{n-n_4 p}).
\end{equation}
Since $\widehat{F}$ does not contain a copy of $G^{p+1}$, we have $e(\widehat{F})\leq e(F_{n-n_4 p})$, where $F_{n-n_4 p}$ is an extremal graph for $G^{p+1}$ on $n-n_4 p$ vertices. By (\ref{lemma 2}) and (\ref{lemma 3}), we have
\begin{align*}
\phi(F_n)&=e(F_n)-e(H_n)=e(T^\prime_\ell)-e(T^\prime_\ell)+(e_F-e_T)+e(\widehat{F})-e(H_{n-n_4 p})\\
&\leq(e_F-e_T)+e(F_{n-n_4 p})-e(H_{n-n_4 p})\\
&=(e_F-e_T)+\phi(F_{n-n_4 p}).
\end{align*}
If $e_F-e_T<0$, then $\phi(F_n)<\phi(F_{n-n_4 p})$. Since $n-n_4p>n/2$, we are done.
Hence, we may assume that $e_F-e_T\geq 0$.
Recall that $n_4=(1-c^2k_1)c^{2\ell}n_3$.
Since $c$ is sufficiently small, we have
\begin{align*}
e_F-e_T\leq& \ell\cdot n_4 p+(n-\ell-n_4 p-|D|)\cdot n_4 (p-1)\\
&+|D|\cdot \left(n_4 (p-2)+(1-c)c^{2\ell}n_3+c^{2\ell+2}n_3\right)\\
&-\left((q-1)\cdot n_4 p+(n-q+1-n_4 p)\cdot n_4 (p-1)\right)\\
\leq&(\ell-(q-1))n_4+(c^2(k_1+1)-c))c^{2\ell}n_3|D|\\
\leq &0,
\end{align*}
with the equality holds if and only if $|D|=0$, $\ell=q-1$, and each vertex of $C^\prime_i $ is adjacent to each vertex of $B^\prime_{j\neq i}$. Note that $T_p(n-q+1-n_4 p)$ has more edges than any other $p$-chromatic graph on $n-q+1-n_4 p$ vertices. It follows from  Lemmas~\ref{lemma for F(k,p)(1)} and \ref{lemma for F(k,p)}, that $e(F_n)\leq h(n,p,q)+f(k-1,k-1)$ with equality holds if and only if $F_n\in \mathcal{H}(n,p,q,k-1,k-1,K_{q})$ (If each graph in $\mathcal{H}(n,p,q,k-1,k-1,K_{q})$ contains a copy of $G^{p+1}$, then we have $e(F_n)<h(n,p,q)+f(k-1,k-1)$). The proof is complete.
\end{proof}

\noindent {\bf Proofs of Theorem~\ref{maintheorem1}(\romannumeral2) and Theorem~\ref{non-bipartite graph}:}\\
\begin{proof} The proof is essentially the same as the proof of Theorem~\ref{maintheorem1}(\romannumeral1) and we sketch the proof as follows. Let $F_n$ be an extremal graph for $G^{p+1}$. Applying Lemma~\ref{main1}(b) and Proposition~\ref{L-critical extremal 2}, we obtained that
\begin{equation}\label{eq for main theorem 2}
e(F_n)\geq h^\prime(n,p,q)+\mbox{ex}(q-1,\mathcal{B}).
\end{equation}
By Theorem~\ref{erdos-stone}, $F_n$ contains a copy of $T_p(n_1p)$. Since $\mathcal{M}$ contains a matching, as previous arguments, by progressive induction (Lemma~\ref{progrssion induction}), $F_n$ can be partitioned into $\cup _{i=1}^p (B^\prime_i\cup C^\prime_i)\cup E$ with $|E|=q-1$ satisfies the following.
\begin{itemize}
  \item  $F_n[\cup _{i=1}^p B^\prime_i]=T_p(n_4p)$;
  \item each vertex of $E$ is adjacent to each vertex of $T_p(n_4p)$;
  \item each vertex of $C^\prime_i $ is adjacent to each vertex of $B^\prime_{j\neq i}$;
  \item  $e(F_n[E])\leq\mbox{ex}(q-1,\mathcal{B})$.
\end{itemize}
It follows from Proposition~\ref{3.9} that $k=1$, and hence there is no edge in $B^\prime_i\cup C^\prime_i$ for $i\in[p]$. Thus, by (\ref{eq for main theorem 2}), the result follows similarly as the proof of Theorem~\ref{maintheorem1}(\romannumeral1).
\end{proof}

\section{\bf Proof of Theorem~\ref{coro5}}\label{section-proof-1.3}

We say a graph is {\it factor-critical} if $\nu(G)=\nu(G-\{v\})=\lfloor |V(G)|/2\rfloor$ for all $v\in V(G)$.
Clearly, a factor-critical $n$-vertex graph has odd number of vertices and a matching on $n-1$ vertices.
We need the following well-known lemma of Gallai.

\begin{lemma}[Gallai \cite{Gallai1963}]\label{gallai-lemma}
If graph $G$ is connected and $\nu(G-\{v\}) = \nu(G)$ for each $v\in V(G)$, then $G$ is factor-critical.
\end{lemma}

Denote by $K_{s,t}(a,b)$ be the graph obtained from taking a copy of $K_{s,t}$ and splitting $a$ and $b$ vertices in the partite sets of $K_{s,t}$  with sizes $s$ and $t$ respectively.

By Theorem~\ref{chvatal}, we have
\begin{equation}\label{eq-of-f(k-1,k-1)}
f(k-1,k-1)= \left\{
                  \begin{array}{ll}
                      k^2-k , & \hbox{for odd $k$;} \\
                     k^2-3k/2, & \hbox{for even $k$.}
                  \end{array}
                \right.
\end{equation}

Define
$$g(k-1,k-1)=\left\{
              \begin{array}{ll}
                (2k^2-3k-1)/2, & \hbox{for odd $k$;} \\
                k^2-2k+1, & \hbox{for even $k$.}
              \end{array}
            \right.$$

\begin{lemma}\label{lemma for Kst even t}
Let $\mathcal{F}=\{S_{t+1},M_{2t},K_{2,t-1}(0,i):i=0,1,\ldots,t-1\}$ with $t\geq 3$. If $n$ is sufficiently large, then
$$\emph{ex}(n,\mathcal{F})=g(t-1,t-1).$$
\end{lemma}
\begin{proof}
Let $G$ be an extremal graph for $\mathcal{F}$.
Clearly, by Theorem~\ref{chvatal}, we have $e(G)\leq f(t-1,t-1)$.
Thus $e(G)$ is a constant depending on $t$.
Now we do not consider the isolated vertices of $G$.
Let $F_1=K_t\cup K_{t-1}$ when $t$ is even and $F_2=K_t\cup (K_t-E(S_3\cup M_{t-3}))$ when $t$ is odd.
Then $F_1$ and $F_2$ are $\mathcal{F}$-free with $g(t-1,t-1)$ edges.
Thus we have $e(G)\geq g(t-1,t-1)$.
From Theorem~\ref{chvatal} and $t\geq 3$, easy calculations show that $f(t-1,t-2)\leq g(t-1,t-1)$ and $f(t-2,t-1)\leq g(t-1,t-1)$.
Hence we may suppose $\Delta(G)= t-1$ and $\nu(G)= t-1$.

Let $s(G)$ be the number of components of $G$ which are stars.
We choose $G$  with $s(G)$ as large as possible.
First we will show that each component of $G$ is either factor-critical or a star.
Suppose for contrary that there is a component $C$ of $G$ which is neither  factor-critical nor a star.
Choose $x\in V(C)$ such that $\nu (C-\{x\})=\nu (C)-1$.
If $d_G(x)=t-1$, then $\Delta(G-\{x\})\leq t-2$, as otherwise $G$ contains a copy of  $K_{2,t-1}(0,i)$ with $0\leq i\leq t-1$.
Thus from (\ref{eq-of-f(k-1,k-1)}), we have
$$e(G-\{x\})\leq \left\{
                  \begin{array}{ll}
                    f(t-2,t-2)=  (t-1)^2-3(t-1)/2   , & \hbox{for odd $t$;} \\
                    f(t-2,t-2)= (t-1)^2-(t-1) , & \hbox{for even $t$.}
                  \end{array}
                \right.$$
Hence $e(G)\leq (t-1)^2-3(t-1)/2+(t-1)\leq  (2t^2-5t+3)/2<(2t^2-3t-1)/2$ for odd $t$ and $e(G)\leq (t-1)^2$ for even $t$.
In both cases, we are done.
Now assume that $d_G(x)\leq t-2$.
Let $G^\prime$ be the graph consisting of vertex-disjoint union of $G-\{x\}$ and a copy of $S_{t-1}$.
Then $G^\prime$ is $\mathcal{F}$-free but with $s(G^\prime)>s(G)$ and $e(G^\prime)\geq e(G)$, a contradiction to our choice of $G$.
Thus  each component of $G$ is either factor-critical or a star.

If $G$ contains an $S_t$-component, then let $G^\prime=G-S_t$.
Since $G$ is $\mathcal{F}$-free, we have $\nu(G^\prime)\leq t-2$ and $\Delta(G^\prime)\leq t-2$.
The result follows similarly as the last paragraph (by (\ref{eq-of-f(k-1,k-1)}) and $e(S_t)=t-1$).
Since $G$ is an extremal graph, we may suppose that each star of $G$ is $S_{t-1}$.

Now, we may suppose that $C_1,C_2,\ldots,C_m$ are the components of $G$  such that $C_i$  are factor-critical.
Moreover, suppose that $\Delta(C_1)=t-1$ and $\Delta(C_i)\leq t-2$ for $i=2,\ldots,m$.
Let $V(C_i)=n_i$ and let $s=s(G)$ be the number of $S_{t-1}$-components of $G$.
Hence, we have
$$e(C_i)\leq \min\left\{{n_i \choose 2},\left\lfloor \frac{(t-2)n_i}{2}\right\rfloor\right\} \mbox{ for }i=2,\ldots,m.$$
Since $n_1\geq  2\lfloor t/2 \rfloor+1$ (by $\Delta(C_1)=t-1$ and $n_1$ is odd), by $\nu(G)= t-1$, we have $n_i\leq 2 \lceil t/2 \rceil-1$.
If $n_i=t$ for some $i=2,\ldots,m$, then $t$ is odd and $m=2$.
Hence $G$ consists of two components with on $t$ vertices.
Since $C_1$ and $C_2$ are factor-critical graphs with $\Delta(C_1)=t-1$ and $\Delta(C_2)=t-2$, we have $e(G)\leq(t(t-1)+(t-2)t-1)/2=(2t^2-3t-1)/2$.
The result follows.
Thus we may suppose $n_i\leq t-1$ and hence $e(C_i)= {n_i \choose 2}$  by the maximality of $G$ for $i=2,\ldots,m$.
Let $\ell$ be the number of vertices in $C_1$ with degree $t-1$.
Since $G$ does not contain a copy of $K_{2,t-1}(0,i)$ for $0\leq i\leq t-1$, the vertices of $G$ with degree $t-1$ form a clique in $G$ and $\ell\leq t$.
Thus, we have
\begin{equation}\label{eq.1}
e(C_1)\leq \min\left\{ {\ell \choose 2}+ \ell(t-\ell)+{n_1-\ell \choose 2}, \left\lfloor\frac{\ell(t-1)+(n_1-\ell)(t-2)}{2}\right\rfloor  \right\}.
\end{equation}
Hence, $$e(G)=\max\left\{e(C_1)+\sum_{i=2}^m{n_i \choose 2}+s(t-2):\sum_{i=1}^{m}{\lfloor n_i/2\rfloor}+s=t-1\right\}.$$
We will estimate $e(G)$ by considering $e(C_i)/\nu(C_i)$ for each component of $G$ in following two cases:

(1). $t$ is odd.
Since $n_i\leq t-2$ for $i=2,\ldots,m$, it is easy to see that $e(G)$ attains its maximum when $n_1=t$ ($\ell=t$), $m=1$ and $s=\lfloor t/2\rfloor$.\footnote{There are other possible values of $n_1$, $m$ and $s$.}
Hence, $e(G)\leq t(t-1)/2+\lfloor t/2\rfloor (t-2)=(t-1)^2< (2t^2-3t-1)/2$, we are done.

(2). $t$ is even.
Then, $e(G)$  attains its maximum when $n_1=2t-1$ and $\ell=t$, or $n_1=t$ and $n_2=t-1$.
Thus  $e(G)\leq t(t-1)/2+(t-1)(t-2)/2=(t-1)^2 $.
The proof of the lemma is complete.
\end{proof}

To establish the lower bound of Theorem~\ref{coro5}, we need the following graphs.
For even $t$ with $t\geq 6$,
let $X=\{x_1,\ldots,x_{t-1}\}$ and $Y=\{y_1,\ldots,y_{t-1}\}$.
Set $X_1=\{x_1,x_2\ldots,x_{t/2-1}\}$, $Y_1=\{y_1,y_2,\ldots,y_{t/2-1}\}$, $X_2= \{x_{t/2},x_{t/2+1},\ldots,x_{t-2}\}$ and $Y_2=\{y_{t/2},y_{t/2+1},\ldots,y_{t-2}\}$.
The graph   $H_{2t-1}$ is obtained as following:
\begin{itemize}
  \item Taking two vertex-disjoint copies of $K_{t-1}$ with vertex sets $X$ and $Y$ respectively;
  \item adding a matching with size $t/2-1$ between $X_1$ and $Y_1$, a cycle of length $t-2$ between $X_2$ and $Y_2$ and the edge $x_{t-1}y_{t-1}$;
  \item adding a new vertex $z$ and joining it to each vertex of $X_1$ and $Y_1$;
  \item deleting a matching with size $t/2-1$ between $X_1$ and $X_2$ and between $Y_1$ and $Y_2$ respectively.
\end{itemize}
For $t=4$, let $H_7$ be graph obtained from vertex-disjoint union of copies of $K_4$ and $K_3$ by deleting an edge of $K_4$ and joining the incident vertices of the deleted edge to $K_3$ by two independent edges.

\begin{proposition}\label{extremal-s}
Let $t\geq 4$ be even and $\mathcal{K}(t)=\{K_{a,b}(0,c):a+b=t+1, \mbox{ and } a\geq 3 \mbox{ or }c=0\}$. Then $H_{2t-1}$ is $\mathcal{K}(t)$-free.
\end{proposition}
\begin{proof}
Clearly $H_{7}$ is $\mathcal{K}(4)$-free.
Let $t\geq 6$.
It is not hard to check that $H_{2t-1}$ does not contain a complete bipartite graph on at least $t+1$ vertices, i.e., $H_{2t-1}$ does not contain a copy of $K_{a,b}(0,0)$ with $a+b=t+1$.
Suppose that $H_{2t-1}$ contains a copy of $K_{a,b}(0,c)$ with $a+b=t+1$, $a\geq 3$ and $c>0$.
Then $H_{2t-1}$ contains a copy of $K_{a,b-c}$.
Let $A$ and $B$ be the classes of $K_{a,b-c}$ with sizes $a$ and $b-c$ respectively.
Clearly, we have
\begin{equation}\label{eq-for-extremal-K_s,t}
|V(K_{a,b}(0,c))|=a+b-c+ac=t+1+(a-1)c.
\end{equation}
If $b-c=1$, then $c=t-a$.
By $a\geq 3$ and $t\geq6$, we have $|V(K_{a,b}(0,c))|=t+1+(a-1)(t-a)\geq 2t>|V(H_{2t-1})|$, a contradiction.
We may suppose that $b-c\geq 2$.
Assume that $t\geq 8$.
Let $z$ be the vertex in definition of $H_{2t-1}$.
We will find contradictions in the following two cases.

\medskip

{\bf Case 1.} $A\subseteq X\cup\{z\}$ or $A\subseteq Y\cup\{z\}.$
Without loss of generality, let  $A\subseteq X\cup\{z\}$.
Since $a\geq 3$, it is obviously that each vertex of $B$ belongs to $X\cup\{z\}$ (any three vertices of $A\cup\{z\}$ have no common neighbours in $Y$).
Note that the complement graph of $H_{2t-1}[X\cup\{z\}]$ is connected.
$H_{2t-1}[X\cup\{z\}]$ contains no complete bipartite graph on at least $t$ vertices.
Thus we  have $a+b-c\leq t-1$, i.e., $c\geq 2$.
Recall that $a\geq 3$.
(a). $c\geq 4$, or $c=3$ and $a\geq 4$.
Then there are at most $t+c+2(a-1)<t+1+(a-1)c$ vertices incident with a copy of $K_{a,b}(0,c)$ with $A\subseteq X\cup\{z\}$, a contradiction .
(b). $c=2$.
Then $a+b-c=t-1$.
If $z\notin V(K_{a,b-c})$, then $V(K_{a,b-c})=X$.
Note that there is a matching  between $X_1$ and $X_2$ in the complement graph of $H_{2t-1}$.
There are at most $\lfloor a/2\rfloor$ vertices of $X_2$ belonging to $A$.
Hence there are at most $t-1+ 2\lfloor a/2\rfloor+ \lceil a/2\rceil<t+1+2(a-1)$ (by $a\geq 3$) vertices incident with a copy of $K_{a,b}(0,c)$ with $A\subseteq X\cup\{z\}$, a contradiction.
If $z\in V(K_{a,b-c})$, then an easy observation shows that $K_{a,b-c}=K_{t-2,1}$, a contradiction to $b-c\geq 2$.
(c) $c=3$ and $a=3$.
Then $a+b-c=t-2$.
Clearly, $A$ consists of $z$ and two vertices of $X_2$, otherwise we have $|V(K_{a,b-c})|< t+3+2+2=t+7$, a contradiction to (\ref{eq-for-extremal-K_s,t}).
Now, since $t\geq 8$, $A$ has at most $t-6$ common neighbors in $X$ and hence  $a+b-c\leq t-3$.
This final contradiction completes our proof for Case 1.
\medskip

{\bf Case 2.} $A\cap X\neq \emptyset$ and $A\cap Y\neq \emptyset$.

\medskip

If $z\in A$, then $|A\cap X|=|A\cap Y|=1$, and hence $a=3$, $b-c=2$ and $c=t-4$.
Thus by (\ref{eq-for-extremal-K_s,t}) and $t\geq 8$, we have $|V(K_{a,b}(0,c))|= t+1+2(t-4)\geq 2t>|V(H_{2t-1})|$, a contradiction.
Let $z\notin A.$
Now, without loss of  generality, we may suppose that $|A\cap X|\geq 2$.
Recall that any three vertices of $X$ have no common neighbors in $Y$.
It follows from $b-c\geq 2$ that $|A\cap X|=2$ and $1\leq |A\cap Y|\leq 2$.
If  $|A\cap Y|=1$, then $a=3$ and $b-c=2$ for $t=8$ and $a=3$ and $2\leq b-c\leq 3$ for $t\geq 10$.
Thus $c=t-4$ for $t=8$ and $c\geq t-5$ for $t\geq 10$.
By (\ref{eq-for-extremal-K_s,t}), we have $|V(K_{a,b}(0,c))|= t+1+2c\geq 2t>|V(H_{2t-1})|$, a contradiction.
If  $|A\cap Y|=2$, then $a=4$, $b-c=2$ and $c=t-5$.
By (\ref{eq-for-extremal-K_s,t}) and $t\geq 8$, we have $|V(K_{a,b}(0,c))|= t+1+3(t-5)=4t-14\geq 2t>|V(H_{2t-1})|$, which is also a contradiction.
The proof of Case 2 is complete.
\medskip

Let $t=6$.
Recall that $b-c\geq 2$.
Consider $K_{a,b-c}=K_{4,2}$, $K_{a,b-c}=K_{3,2}$ and $K_{a,b-c}=K_{3,3}$ respectively, it is easy to see that  $H_{11}$ is $\mathcal{K}(6)$-free.
The proof is complete.
\end{proof}

Now, we are ready for the proof of Theorem~\ref{coro5}.

\medskip

\noindent{\bf Proof of Theorem~\ref{coro5}:}

\begin{proof}
Let $s\leq t$.
If $t\leq 2$, then $K_{s,t}$ is a star or $C_4$.
The result follows by Corollaries~\ref{coro2} and~\ref{coro4}.
Thus we can assume that $t\geq 3$.
Let $F_1\in \mathcal{H}(n,p,s,t-1,t-1,S_{s-1})$  be the graph obtained from taking a copy of $H^\prime(n,p,s)$, embedding a copy of $H_{2t-1}$ when $t$ is even (two vertex-disjoint copies of $K_{t}$ when $t$ is odd respectively.) in one class of the copy of $T(n-s+1)$ in $H^\prime(n,p,s)$ and embedding an extremal graph for $S_{s-1}$ on $s-1$ vertices in the copy of $K_{s-1}$ in $H^\prime(n,p,s)$.
Let $F_2$ be the graph obtained from  taking a copy of $H(n,p,s)$ and embedding an extremal graph in Lemma~\ref{lemma for Kst even t} into one class of the copy of $T(n-s+1)$ in $H(n,p,s)$.

Lemma~\ref{liu lemma} implies that $\mathcal{M}(K_{s,t}^{p+1})=\{K_{s,t}(i,j):0\leq i\leq s, 0\leq j\leq t\}$.
Careful  observation shows that after deleting any  $s-1$ vertices, say $X$, of any $M\in \mathcal{M}(K_{s,t}^{p+1})$ with $S_{s-1}\nsubseteq M[X]$, the obtained graph satisfies one of the following:
\begin{itemize}
  \item  contains a copy of $M_t$;
  \item  contains a copy of $S_t$;
  \item  contains a copy of $K_{a,b}$ with $a+b=t+1$;
  \item  contains a copy of $K_{a,b}(0,c)$ with $a+b=t+1$, where $a\geq 3$  when $s\neq t-1$, and  $a\geq 2$ when $s= t-1$; or
  \item  has at least $2t$ vertices.
  \end{itemize}

Thus, by Proposition~\ref{L-critical extremal 2}, if $s\neq t-1$, then both $F_1$ and $F_2$ are $K_{s,t}^{p+1}$-free; if $s=t-1$, then $F_2$ is $K_{s,t}^{p+1}$-free.
Hence, we have $e(G)\geq \max\{e(F_1),e(F_2)\}$ and if $s=t-1$, then $e(G)\geq  e(F_2)$

Let $F_n$ be an extremal graph for $K_{s,t}^{p+1}$.
Then, by the last paragraph, we have
\begin{equation}\label{eq-for-K-s-t}
e(F_n)\geq h(n,p,s)+f(t-1,t-1)-\left\lceil\frac{s-1}{2}\right\rceil+i,
\end{equation}
where $i=1$ if $s=t$ is even and $i=0$ otherwise.
By Theorem~\ref{erdos-stone}, $F_n$ contains a copy of $T_p(n_1p)$. Since $\mathcal{M}$ contains a matching, as previous arguments in Section~\ref{section5}, by progressive induction (Lemma~\ref{progrssion induction}), $F_n$ can be partitioned into $\cup _{i=1}^p (B^\prime_i\cup C^\prime_i)\cup E$ with  $|B^\prime_i\cup C^\prime_i|=n_i^\prime$  and $|E|=s-1$ satisfies the following.
\begin{itemize}
  \item  $F_n[\cup _{i=1}^p B^\prime_i]=T_p(n_4p)$ is an induced subgraph of $F_n$;
  \item each vertex of $E$ is adjacent to each vertex of $F_n[\cup _{i=1}^p B^\prime_i]$;
  \item each vertex of $C^\prime_i $ is adjacent to each vertex of $B^\prime_{j\neq i}$.
\end{itemize}
If two of $E(F_n[C^\prime_i])$ are non-empty, then by (\ref{equation for K_s,t}) in Lemma~\ref{lemma for F(k,p)}, we have $e(F_n)\leq h(n,p,s)+t^2-2t$, a contradiction to (\ref{eq-for-K-s-t}).
Thus, we may assume that $E(F_n[B^\prime_i\cup C^\prime_i])$ are empty for $i=2,\ldots,p$.
Hence by Proposition~\ref{L-critical extremal 2}, we have $e(F_n[B^\prime_1\cup C^\prime_1])\leq \mbox{ex}(n^\prime_1+s-1,\mathcal{M}(K_{s,t}^{p+1}))$.

\medskip

{\bf Claim.} Let $p(s,t)=\mbox{ex}(n,\mathcal{M}(K_{s,t}^{p+1}))-h^\prime(n,1,s)$.
Then there exists an $n_0$ such that $p(s,t)$ depends on $s$ and $t$, when $n\geq n_0$.

\medskip

\begin{proof}
Let $F$ be an extremal graph for $\mathcal{M}(K_{s,t}^{p+1})$ and $n\geq n_0$.
Clearly, $H(n,1,s)$ is $\mathcal{M}(K_{s,t}^{p+1})$-free, and hence $e(F)\geq (s-1)(n-s+1)$.
Note that $\mathcal{M}(K_{s,t}^{p+1})$ contains $s$ vertex-disjoint copies of $S_{t+1}$.
There is a set of vertices $X$ such that each vertex in $X$ has degree at least $n-c(s,t)$, where $c(s,t)$ is a constant.
Let $Y$ be the isolated vertices of $V(F)-X$ and $Z=V(F)-X-Y$.
Note that there is a large complete bipartite graph between $X$ and $Y$ (by $n\geq n_0$).
Since $F$ is an extremal graph, $F[X,Y]$ is  a  complete bipartite graph.
Let $p(s,t)=e(F[X])+e(F[Z])-(|X||Z|-e(F[X,Z]))$.
The result follows since $e(F[X])$, $e(F[Z])$ $e(F[X,Z])$, $|X|$ and $|Z|$ depend on $s$ and $t$.
\end{proof}

Thus, since $n_1$ is sufficiently large, it follows from the claim that
\begin{equation}\label{eq for theorem 1.3}
e(F_n)\leq \max\left\{(s-1)(n-s+1)+p(s,t)+  \sum_{i\neq j} n^\prime_i n_j^\prime:\sum_{i=1}^p n_i^\prime=n-s+1   \right\}.
\end{equation}

Embedding an extremal graph for $\mathcal{M}(K_{s,t}^{p+1})$ on $n_1$ vertices in the   partite set of $K_{n_1,\ldots,n_p}$ with size $n_1$, where $n_1=\lceil(n-s+1)/p\rceil+s-1$ and $n_i\in\{\lceil(n-s+1)/p\rceil,\lfloor(n-s+1)/p\rfloor\}$, by Proposition~\ref{L-critical extremal 2}, we have
$$
\mbox{ex}(n,K_{s,t}^{p+1})\geq \mbox{ex}(n_1,\mathcal{M}(K_{s,t}^{p+1}))+ \sum_{i\neq j}n_in_j=h^\prime(n,p,s)+p(s,t).
$$
Combining with (\ref{eq for theorem 1.3}), we have $\mbox{ex}(n,K_{s,t}^{p+1})=h^\prime(n,p,s)+p(s,t)$.
The proof of Theorem~\ref{coro5} is complete.
\end{proof}

\section{Conclusion}\label{conclusion}
By similar method as the proof of Theorem~\ref{maintheorem1}, one can obtain the following corollary.

\begin{corollary}\label{matching2}
Let $\mathcal{F}$ be a family of graphs with $p(\mathcal{F})=p\geq 2$. If $\mathcal{M}(\mathcal{F})=\{M_{2s}\}$, then $$\emph{ex}(n,\mathcal{F})=h(n,p,s),$$
provided $n$ is sufficiently large. Moreover, the unique graph in $\mathcal{H}(n,p,s,0,0,K_s)$ is the extremal graph for $\mathcal{F}$.
\end{corollary}

\remark There are many interesting graphs that belong to the graphs set of Corollary~\ref{matching2}, including the Petersen graph, vertex-disjoint union of cliques with same order and the dodecahedron $D_{20}$.

\medskip

In Theorem~\ref{coro5}, we do not determine  $\mbox{ex}(n,\mathcal{M}(K^{p+1}_{s,t}))$. So we may pose the following conjecture.
\begin{conjecture} Let $n\geq s+2t$ and $p\geq 3$. Then
\begin{equation*}
\emph{ex}(n,\mathcal{M}(K^{p+1}_{s,t}))=h(n,1,s)+f(t-1,t-1)-\left\lceil\frac{s-1}{2}\right\rceil+i,
\end{equation*}
where $i=1$ if $s=t$ is even and $i=0$ otherwise.
\end{conjecture}

Now we will find a large family of non-bipartite graphs which are counter-examples for Conjecture~\ref{Keevash-Sudakov-conjecture} \footnote{Chen, Ma, Qiu and Yuan \cite{chen2020} independently found different counter-examples for Conjecture~\ref{Keevash-Sudakov-conjecture}.}.  First, we need a lemma  proved in \cite{Yuan2020} (also see Lemma 3.2 in \cite{keevash2004}). If an edge $e$ is not contained in any monochromatic copy of a given graph $H$, then we call $e$ being  NIM-$H$. If a subgraph $G$ of $K_n$ consists of NIM-$H$ edges, then we call $G$ being NIM-$H$.

\begin{lemma}\label{Lemma-for-monochromatic-T(n,p)}
Let $p\geq 2$, $t\geq 1$ be given integers and $H$ be a given graph. Then there exists an integer $n_0=n_0(p,t,H)$ such that if $G$ is an NIM-$H$ graph on $n\geq n_0$ vertices containing at least $t_p(n)$ edges, then $G$ contains a blue (or red) copy of $T_p(tp)$ such that the edges inside each class are red (or blue), where $t\geq |V(H)|$.
\end{lemma}

Now we present the following theorem disproving Conjecture~\ref{Keevash-Sudakov-conjecture}.

\begin{theorem}\label{conter-example}
Let $G$ be a bipartite graph with $q<|A|$ or a non-bipartite graph with $2\leq\chi(G)\leq p-1$. Let $n$ be sufficiently large. Then the following hold:\\
(a). If $\mathcal{B}=\{K_q\}$, then
$$g(n,G^{p+1})=\emph{ex}(n,G^{p+1}).$$
(b). If $\mathcal{B}\neq\{K_q\}$, then
$$g(n,G^{p+1})=\emph{ex}(n,G^{p+1})+{q-1 \choose 2}-\emph{ex}(q-1,\mathcal{B}).$$
In particular, we have
$$g(n,K_{t}^{p+1})=\emph{ex}(n,K_{t}^{p+1})+{{t-1 \choose 2} \choose 2}.$$
\end{theorem}

\begin{proof}
The proof is essentially the same as the proof of Theorem~\ref{maintheorem1}(\romannumeral2) and Theorem~\ref{non-bipartite graph}.
Hence we sketch the proof as follows.
Let $F_n$ be an extremal NIM-$G^{p+1}$ graph. Clearly, we have $e(F_n)\geq \mbox{ex}(n,G^{p+1})\geq t_p(n)$. Without loss of generality, by Lemma~\ref{Lemma-for-monochromatic-T(n,p)}, $F_n$ contains  a blue copy of $T_p(t_0p)$ such that the edges inside each class are red, where $t_0$ is sufficiently large. Since $\mathcal{M}$ contains a matching, as previous arguments, by Proposition~\ref{3.9}, Lemmas~\ref{main1}(b) and \ref{progrssion induction}, $F_n$ can be partitioned into  $(\cup _{i=1}^p (B_i\cup C_i))\cup E$ with $|E|=q-1$ satisfying the following;
\begin{itemize}
  \item  $F_n[\cup _{i=1}^p B_i]=T_p(tp)\subseteq T_p(t_0p)$, where $t\leq t_0$ is sufficiently large;
  \item each vertex of $E$ is adjacent to each vertex of $T_p(t p)$ by a blue edge in $F_n$;
  \item each vertex of $C_i $ is adjacent to each vertex of $B_{j\neq i}$ (not necessary colored by blue);
  \item there is no NIM-edge (blue edge) in $F_n[ B_i\cup C_i]$ for $i\in [p]$;
  \item the blue edges inside $E$ are at most $\mbox{ex}(q-1,\mathcal{B})$.
\end{itemize}
Thus, we have
\begin{equation}\label{for-counter-example}
e(F_n)\leq h^\prime(n,p,q)+e(F_n[E])\leq  h(n,p,q).
\end{equation}
We divide the following proof into two cases.

\medskip

(a). $\mathcal{B}=\{K_q\}$.

\medskip

Clearly, by Theorem~\ref{maintheorem1}(\romannumeral2) and Theorem~\ref{non-bipartite graph}, we have $g(n,G^{p+1})\geq \mbox{ex}(n,G^{p+1})=h(n,p,q)$.
By $|E|=q-1$ and (\ref{for-counter-example}), we have $e(F_n)\leq h(n,p,q)$, and hence $g(n,G^{p+1})=\mbox{ex}(n,G^{p+1})$,  the result follows.

\medskip

(b). $\mathcal{B}\neq\{K_q\}$.

\medskip

Fist, we present lower bound of $g(n,G^{p+1})$.
We color a copy of $H^\prime(n,p,q)$ in $K_n$ blue and color other edges red.
Clearly, each blue edge is an NIM-edge.
Note that $q-1< |V(G^{p+1})|$.
The red edges inside $\overline{K}_{q-1}$ of $H^\prime(n,p,q)$ are NIM-edges.
Thus we have $g(n,G^{p+1})\geq h(n,p,q)$.
Combining with (\ref{for-counter-example}), we have $g(n,G^{p+1})=h(n,p,q)$, i.e, $g(n,G^{p+1})=\mbox{ex}(n,G^{p+1})+{q-1 \choose 2}-\mbox{ex}(q-1,\mathcal{B}).$
Note that ${q-1 \choose 2}-\mbox{ex}(q-1,\mathcal{B})>0$ when $\mathcal{B}\neq\{K_q\}$.
Thus we find counter-examples for Conjecture~\ref{Keevash-Sudakov-conjecture}.
\end{proof}

Basing on Theorem~\ref{conter-example}, we may pose the revised conjecture.
\begin{conjecture}
Let $n$ be sufficiently large. Then there exist a constant $h_G$ depending on $G$ such that $g(n,G)=\emph{ex}(n,g)+h_G$.
\end{conjecture}

\noindent {\bf Acknowledgement.}
The author would like to thank the referee for his/her suggestions and Xing Peng for his careful reading on a draft.

\newpage
\appendix
\section{Proof of Lemma~\ref{Lemma-for-monochromatic-T(n,p)}}

We need the following well-known theorems. The first one due to Ramsey is one of the most important results in combinatorics.

\begin{theorem}[Ramsey]\label{Ramsey} For every $t$ there exists $N = R(t)$ such that every $2$-coloring of the edges of
$K_N$ has a monochromatic $K_t$ subgraph.
\end{theorem}

In 1954, K\"{o}v\'{a}ri, S\'{o}s, and Tur\'{a}n proved the following theorem.

\begin{theorem}[K\"{o}v\'{a}ri, S\'{o}s, and Tur\'{a}n].\label{KST}
$$\emph{ex}(n,K_{t,t})=O(n^{2-\frac{1}{t}}). $$
\end{theorem}

Now we are ready to prove Lemma~\ref{Lemma-for-monochromatic-T(n,p)}.\\
\noindent\begin{proof} Since $G$ is an NIM-$H$ graph on $n$ vertices containing at least $t_p(n)$ edges, by  Theorem~\ref{erdos-stone}, $G$ contains $T_p(Np)$ as a subgraph with a vertex partition $V_1\cup \ldots \cup V_p$, where $N$ is a large constant depending on Theorems~\ref{Ramsey},~\ref{KST}, and $p$.
\medskip

\noindent{\bf Claim 1.} There exists a constant $m(t)$ depending on $t$ such that for any two disjoint vertex sets $U,V$ of $K_n$ with $|V|=|U|=m(t)$, there is a monochromatic copy of $K_{t,t}$ between $U$ and $V$.

\medskip

\begin{proof} Without lose of generality, suppose that there are at least $\frac{1}{2}m(t)^2$ red edges between $U$ and $V$. Since $\frac{1}{2}m(t)^2\geq O(2m(t))^{2-\frac{1}{t}}$ when $m(t)$ is large, the result follows from Theorem~\ref{KST}.\end{proof}

\noindent{\bf Claim 2.}  Any two monochromatic copies of $K_\ell$ with $\ell\geq m(t)$ in different classes of $T_p(Np)$ have the same color.

\medskip

\begin{proof} Let $\ell\geq m(t)$. Suppose that there are a red copy of $K_\ell$ in $V_1$ and a blue copy of $K_\ell$ in $V_2$. Then it follows from Claim 1 that there is a monochromatic copy of $K_{t,t}$ between the red $K_\ell$ and the blue $K_\ell$. Since $t\geq |V(H)|$, the edges of $K_{t,t}$ are contained in a monochromatic copy of $H$, contradicting that $G$ is an NIM-$H$ graph. The proof is complete.\end{proof}

Applying Theorem~\ref{Ramsey}, there is monochromatic copy of $K_{m_0(t)}$ in each class of $T_p(Np).$ By Claim 2, those $p$ monochromatic copies of $K_{m_0(t)}$ have the same color, say red. Let $G^\prime$ be the subgraph of $G$ induced by those $p$ copies of $K_t$. Then $G^\prime$ has a vertex partition $V(G^\prime)=V^\prime_1\cup \ldots \cup V^\prime_p$ such that $V_i^\prime\subset V_i$ for each $i\in[p]$. We say a pair of disjoint vertex sets is monochromatic (red/blue) if all the edges between them have the same color (red/blue).

We will find a copy of $T_p(tp)$ with the property what we need by defining a sequence of graphs. Let $G^0=G^\prime$. By Claim 1, we can define $G^{i+1}$ from $G^i$ by taking $m_{i+1}(t)$ vertices of each class of $G^i$ such that the number of  monochromatic  pairs in the classes of $G^{i+1}$ is least $i+1$, where $m_{i+1}(t)<m_{i}(t)$ is a sufficiently large constant depending Theorem~\ref{KST}. Note that there are ${p \choose 2}$ pairs of vertex sets between $V_1,\ldots,V_p$. Since the edges between different classes of $G^{p \choose 2}$ are NIM-edges, $G^{p \choose 2}$ has the property needed in the lemma with $t=m_{p \choose 2}(t)$. The proof is complete.\end{proof}

\begin{thebibliography}{1}
\bibitem{abbott1972} H.~L.~Abbott, D.~Hanson, and H.~Sauer, Intersection theorems for systems of sets, \emph{J. Combin. Theory Ser. A} {\bf 12} (1972), 381-389.
\bibitem{Balachandran2009} N.~Balachandran and N.~Khare, Graphs with restricted valency and matching number, \emph{Discrete Math.} {\bf 309} (2009), 4176-4180.
\bibitem{chen2020} C.~Chen, J.~Ma, Y.~Qiu, and L.~Yuan, Edges not in monochromatic copies of a non-bipartite graph, manuscript.
\bibitem{chen2003} G.~Chen, R.~J.~Gould, F.~Pfender, and B.~Wei, Extremal graphs for intersecting cliques, \emph{J. Combin. Theory Ser. B} {\bf 89} (2003), 159-171.
\bibitem{chvatal1976} V.~Chv\'{a}tal and D.~Hanson, Degrees and matchings, \emph{J. Combin. Theory Ser. B} {\bf 20} (1976), 128-138.
\bibitem{erdHos1946} P.~Erd\H{o}s and A.~H.~Stone, On the structure of linear graphs, \emph{Bull. Amer. Math.} {\bf 52} (1946), 1089-1091.
\bibitem{erdHos1959maximal} P.~Erd\H{o}s and T.~Gallai,  On maximal paths and circuits of graphs, \emph{Acta Math. Acad. Sci. Hungar.}  {\bf 10 (3)} (1959),  337-356.
\bibitem{erdHos1962} P.~Erd\H{o}s, \"{u}ber ein Extremalproblem in der Graphentheorie (German), \emph{Arch. Math. (Basel)} {\bf 13} (1962) 122-127.
\bibitem{erdHos1966} P.~Erd\H{o}s and M.~Simonovits, A limit theorem in graph theory, \emph{Studia Sci. Math. Hungar.} {\bf 1} (1966) 51-57.
\bibitem{erdHos1995} P.~Erd\H{o}s, Z. F\"{u}redi, R.~J.~Gould, and D.~S.~Gunderson, Extremal graphs for intersecting triangles, \emph{J. Combin. Theory Ser. B} {\bf 63 (1)} (1995) 89-100.
\bibitem{Furedi2013} Z.~F\"{u}redi and M.~Simonovits, The history of degenerate (bipartite) extremal graph problems. Erd\H{o}s centennial, 169-264, \emph{  Bolyai Soc. Math. Stud.} 25,  J\'{a}nos Bolyai Math. Soc., Budapest, 2013.
\bibitem{Gallai1963} T.~Gallai, Neuer Beweis eines Tutte'schen Satzes, \emph{Magyar Tud. Akad. Mat. Kutat\'{o} Int. K\"{o}zl.} {\bf 8} (1963), 135-139.
\bibitem{glebov2011} R.~Glebov, Extremal graphs for clique-paths, arXiv:1111.7029v1, (2011).

\bibitem{hou2016}  X.~Hou, Y.~Qiu, and B.~Liu, Extremal graph for intersecting odd cycles. \emph{Electron. J. Combin.} {\bf 23} (2016), no. 2, Paper 2.29, 9 pp.

\bibitem{keevash2004} P.~Keevash and B.~Sudakov, On the number of edges not covered by monochromatic copies of a fixed graphs, \emph{J. Combin. Theory Ser. B} {\bf 108} (2004), 41-53.
\bibitem{Liu2013} H.~Liu, Extremal graphs for blow-ups of cycles and trees, \emph{Electron. J. Combin.} {\bf 20 (1)} (2013), Paper 65.

\bibitem{Liu-Pikhurko-Sharifzadeh-2019} H. Liu, O. Pikhurko, and M. Sharifzadeh, Edges not in any monochromatic copy of a fixed graph, \emph{J. Combin. Theory Ser. B} {\bf 135} (2019), 16-43.
\bibitem{Ma-2017} J. Ma, On edges not in monochromatic copies of a fixed bipartite graph, \emph{J. Combin. Theory Ser. B} {\bf 132} (2017), 240-248.

\bibitem{moon1968} J.~W.~Moon, On independent complete subgraphs in a graph, \emph{Canad. J. Math.} {\bf 20} (1968), 95-102.

\bibitem{ni2020} Z.~Ni, L.~Kang, E.~Shan, and H.~Zhu, Extremal graphs for blow-ups of keyrings, \emph{Graphs Combin.} {\bf 36} (2020), no. 6, 1827-1853.


\bibitem{Simonovits1968} M.~Simonovits, A method for solving extremal problems in graph theory, stability problems. In \emph{Theory of Graphs (Proc. Colloq., Tihany, 1966),} Academic Press, New York (1968), 279-319.
\bibitem{Simonovits1974} M.~Simonovits, Extremal graph problems with symmetrical extremal graphs. Additionnal chromatic conditions, \emph{Discrete Math.} {\bf 7} (1974), 349-376.
\bibitem{Simonovits1974.1} M.~Simonovits, The extremal graph problem of the icosahedron, \emph{J. Combin. Theory Ser. B} {\bf 17} (1974), 69-79.
\bibitem{turan1941} P.~Tur\'{a}n, On an extremal problem in graph theory (in Hungrarian), \emph{Mat. es Fiz. Lapok.} {\bf 48} (1941), 436-452.
\bibitem{Wang2020} A.~Wang, X.~Hou, B.~Liu, and Y.~Ma,  The Tur\'{a}n number for the edge blow-up of trees,  arXiv:2008.09998v1.
\bibitem{Yuan2020} L.~Yuan, Extremal functions for odd wheels, submitted, (2020).
\bibitem{Yuan2018} L.~Yuan, Extremal graphs for the $k$-flower, \emph{J. Graph Theory} {\bf 89 (1)} (2018), 26-39.
\bibitem{zhu2020} H.~Zhu, L.~Kang, and E.~Shan, Extremal graphs for odd-ballooning of paths and cycles, \emph{ Graphs Combin.} {\bf 36} (2020), no. 3, 755-766.

\end{thebibliography}
\end{document}